%% file: Topological_Sigma_invariants_Homological_perspective.tex
\documentclass[a4paper,10pt]{article}
\usepackage{amsmath,amssymb,amsfonts,amsthm,mathrsfs}
\usepackage{hyperref}
\usepackage{xcolor}
\usepackage{tikz-cd}
\usetikzlibrary{decorations.pathmorphing} 
\usepackage[all]{xy}
\usepackage[]{algorithm2e}
\usepackage{enumitem}
\usepackage{setspace}

\usepackage{import}
\graphicspath{{Figures/}}
\usepackage[margin=20pt]{caption} 

 \usepackage[margin=1.5 in]{geometry}
  \usepackage[style=alphabetic,sorting=nyt,maxbibnames=7,maxcitenames=5, backend=biber, style=alphabetic]{biblatex}
\newtheorem{thm}{Theorem}[section]

\newtheorem{cor}[thm]{Corollary}
\newtheorem{lem}[thm]{Lemma}
\newtheorem{prop}[thm]{Proposition}
\newtheorem{thma}{Theorem}

\newtheorem{propa}[thma]{Proposition}
\newtheorem{lema}[thma]{Lemma}

\newtheorem*{claim*}{Claim}

\theoremstyle{definition}
\newtheorem{dfn}[thm]{Definition}
\newtheorem{rem}[thm]{Remark}

\renewcommand{\tilde}{\widetilde}

\newcommand{\RR}{\mathbb{R}}
\newcommand{\NN}{\mathbb{N}}
\newcommand{\ZZ}{\mathbb{Z}}

\newcommand{\E}{\operatorname{E}} 
\newcommand{\F}{\operatorname{F}} 

\newcommand{\Ftp}[1]{\mathrm{F}_{#1}}
\newcommand{\Flg}[1]{\mathrm{FP}_{#1}}
\newcommand{\Ctp}[1]{\mathrm{C}_{#1}}
\newcommand{\Clg}[1]{\mathrm{CP}_{#1}}
\newcommand{\TopS}{\Sigma_\mathrm{top}}

\newcommand{\Hom}{\operatorname{Hom}}
\newcommand{\Int}{\operatorname{Int}} 
\newcommand{\TopHom}{\operatorname{Hom}_\mathbf{TopGr}}

\newcommand{\id}{\mathrm{id}}

\newcommand{\VR}[2]{\operatorname{VR}_{#1}({#2})}
\newcommand{\diff}[1]{\partial_{#1}} 
\newcommand{\chains}[3]{\operatorname{C}_{#1}(\VR{#2}{#3})} 
\newcommand{\homology}[3]{\operatorname{H}_{#1}(\VR{#2}{#3})} 
\newcommand{\chain}{\operatorname{C}}
\newcommand{\hlgy}{\operatorname{H}}

\DeclareMathOperator{\Tor}{Tor}
\DeclareMathOperator{\im}{im}
\DeclareMathOperator{\coker}{coker}
\DeclareMathOperator{\supp}{supp}
\DeclareMathOperator{\capacity}{cap}
\DeclareMathOperator{\diam}{diam}

\DeclareMathOperator{\Ind}{Ind}

\title{Geometric invariants of locally compact groups:\\the homological perspective}
\author{Kai-Uwe Bux$^1$ \and Elisa Hartmann$^1$ \and José Pedro Quintanilha$^2$}

\date{$^1$Universität Bielefeld\\
	$^2$Ruprecht-Karls-Universität Heidelberg\\\bigskip
	\today}

\begin{document}

\maketitle

\begin{abstract}
In this paper, we develop the homological version of $\Sigma$-theory for locally compact Hausdorff groups, leaving the homotopical version for another paper. Both versions are connected by a Hurewicz-like theorem. They can be thought of as directional versions of type $\mathrm{CP}_m$ and type $\mathrm{C}_m$, respectively. The classical $\Sigma$-theory is recovered if we equip an abstract group with the discrete topology. This paper provides criteria for type $\mathrm{CP}_m$ and homological locally compact $\Sigma^m$. Given a short exact sequence with kernel of type $\mathrm{CP}_m$, we can derive $\Sigma^m$ of the extension on the sphere that vanishes on the kernel from the quotient, and likewise. Given a short exact sequence with abelian quotient, the $\Sigma$-theory of the extension can tell whether the kernel is of type $\mathrm{CP}_m$.
\end{abstract}

\tableofcontents

\begin{spacing}{1.1}
 \input{intro2}
 \input{recap}
 \input{def+basics}
 \input{classical2}
 \input{criteria}
 \input{ge1}
 \input{ge2}

\end{spacing}
 
\printbibliography

\textsc{Kai-Uwe Bux}, Fakultät für Mathematik, Universität Bielefeld, Postfach 100131, Universitätsstraße 25, D-33501 Bielefeld, Germany\\
\textit{E-mail: }
\texttt{\href{mailto:bux@math.uni-bielefeld.de}{bux@math.uni-bielefeld.de}}\medskip

\textsc{Elisa Hartmann}, Fakultät für Mathematik, Universität Bielefeld, Postfach 100131, Universitätsstraße 25, D-33501 Bielefeld, Germany\\
\textit{E-mail: }
\texttt{\href{mailto:ehartmann@math.uni-bielefeld.de}{ehartmann@math.uni-bielefeld.de}}\medskip

\textsc{José Pedro Quintanilha}, Institut für Mathematik IMa, Im Neuenheimer Feld 205, 69120 Heidelberg, Germany\\
\textit{E-mail: }
\texttt{\href{mailto:jquintanilha@mathi.uni-heidelbeg.de}{jquintanilha@mathi.uni-heidelbeg.de}}

\end{document}

%% file: intro2.tex
\section{Introduction}

For locally compact Hausdorff groups, Abels--Tiemeyer \cite{Abels1997} introduced compactness properties such as being of type $\Ctp m$, the homotopical version, and being of type $\Clg m$, the homological version. Those properties can be thought of as more general versions of the finiteness properties type $\Ftp m$ and type $\Flg m$, which appear as type $\Ctp m$ or type $\Clg m$ if the group is endowed with the discrete topology.

This work and \cite{Hartmann2024c} present locally compact versions of higher geometric invariants $\Sigma^{m}(G)$ and $\Sigma^{m}(G;\ZZ)$. In accordance with Kochloukova \cite{Kochloukova2004}, we name them $\TopS^m(G)$ and $\TopS^m(G;\ZZ)$, the homotopical version and the homological version, respectively. These are collections of continuous homomorphisms $\chi:G\to \RR$, where $G$ is locally compact Hausdorff and is usually also of type $\Ctp m$ or of type $\Clg m$, respectively.

We divided the study on directional versions of the compactness properties $\TopS^m$ into two parts: The homotopical version is treated in \cite{Hartmann2024c} and the homological version is treated in this paper. An immediate consequence of the definition is the following.

\begin{propa}
\label{propa:veryeasy}
 If $G$ is a locally compact, Hausdorff group and $m\ge1$, then the following are equivalent:
\begin{enumerate}
 \item $0\in \TopS^m(G;\ZZ)$;
 \item $\TopS^m(G;\ZZ)\not=\emptyset$;
 \item $G$ is of type $\Clg m$.
\end{enumerate}
\end{propa}
A proof of Proposition~\ref{propa:veryeasy} can be found in Proposition~\ref{prop:veryeasy}. Another easy result is the following.
\begin{propa}
\label{propa:center}
 Let $m\ge 1$. If $G$ is a group of type $\Clg m$ and $\chi:G\to \RR$ a character that does not vanish on the center, then $\chi\in\TopS^m(G;\ZZ)$.
\end{propa}
A proof for Proposition~\ref{propa:center} can be found in Proposition~\ref{prop:center}.

One of our main results is that we can compare homotopical and homological $\TopS^m$.

\begin{thma}
\label{thma:hurewicz}
 If $G$ is a locally compact Hausdorff group, then
 \[
  \TopS^1(G)=\TopS^1(G;\ZZ)
 \]
and if $m\ge 2$ then
\[
 \TopS^m(G)=\TopS^2(G)\cap \TopS^m(G;\ZZ).
\]
\end{thma}
See \cite[Satz~B]{Renz1988} for the classical result. A proof for Theorem~\ref{thma:hurewicz} can be found in Lemma~\ref{lem:sigma1}, Theorem~\ref{thm:hurewicz} and Corollary~\ref{cor:hurewicz}.

We, of course, also would like to compare the new $\Sigma$-invariants with the ones that have been defined by Bieri--Neumann--Strebel--Renz \cite{Bieri1988,Renz1988,Strebel2012}. They appear if an abstract group is equipped with the discrete topology. We call such groups \emph{discrete groups} and $\Sigma$-theory on discrete groups \emph{classical $\Sigma$-theory}.
\begin{thma}
\label{thma:classical}
 If a group $G$ of type $\Flg m$ is endowed with the discrete topology, then \[
 \TopS^m(G;\ZZ)=\Sigma^m(G;\ZZ).
 \]
\end{thma}
Theorem~\ref{thma:classical} is a direct consequence of Theorem~\ref{thm:mainresult}.

A locally compact group $G$ is assigned a sequence of $G$-invariant simplicial sets $(\VR k G)_{k\in\ZZ_{\ge0}}$, the precise definition of which can be found in Section~\ref{sec:def+basics}. Then simplicial $q$-chains $\chain_q(\VR k G)$ of $\VR k G$ are considered as $G$-modules. The chain complexes $(\chain_q(\VR k G),\partial_q)_k$ form a filtration of the standard resolution $(\ZZ[G^{q+1}],\partial_q)$ of the constant $G$-module $\ZZ$. We find it convenient to think of $\chain_*(\VR k G)$ as the graded $\ZZ[G]$-module or $\ZZ$-module $\bigoplus_{q\ge0}\chain_q(\VR k G)$ and we write $\chain_*(\VR k G)^{(m)}$ for the \emph{$m$-skeleton} 
\[
\chain_*(\VR k G)^{(m)}:=\bigoplus_{q=0}^m\chain_q(\VR k G).
\]
It is often more useful to define properties on morphisms than on objects, so we reduced the properties type $\Clg m$ and $\TopS^m(\cdot,\ZZ)$ to properties of chain endomorphisms.
\begin{thma}
\label{thma:mu}
A locally compact Hausdorff group $G$ is of type $\Clg m$ if and only if there exists $K_0\ge 0$ and for every $k\ge K_0$ a finitely modeled $\ZZ G$-chain endomorphism 
\[
\mu_*:{\chains * k G}^{(m)}\to {\chains * {K_0} G}^{(m)}
\]
extending the identity on $\ZZ$.
\end{thma}
In this way, a criterion for type $\Clg m$ is a chain endomorphism sending everything to one index. Theorem~\ref{thma:mu} can be compared with \cite[Theorem~A]{Hartmann2024a}. However, not every chain endomorphism $\mu$ that maps everything to one index is a witness for type $\Clg m$. We need $\mu$ to be finitely modeled; the precise definition can be found in Section~\ref{sec:criteria}. A proof of Theorem~\ref{thma:mu} can be found in Theorem~\ref{thm:mu}.

A continuous homomorphism $\chi:G\to \RR$ is also called \emph{a character on $G$}. Note that this is not a redefinition of the usual notion of a character as a group homomorphism $\chi:G\to \RR$ between abstract groups. Indeed, the category of groups can be fully faithfully embedded in the category of topological groups (whose morphisms are continuous homomorphisms) by endowing an abstract group with the discrete topology, and a homomorphism between abstract groups is sent to the same mapping that is still a homomorphism and continuous since the domain has the discrete topology. Also, the same argument works if the codomain (for example $\RR$) is not discrete.

If $\chi:G\to \RR$ is a character, then every free $G$-module, in particular $\chain_q(\VR k G)$, is equipped with a canonical valuation extending $\chi$.
\begin{thma}
\label{thma:varphi}
 If $G$ is a locally compact Hausdorff group of type $\Clg m$ and $\chi:G\to \RR$ a nonzero character then $\chi\in \TopS^m(G;\ZZ)$ if and only if for every $k\gg 0$ there is $K\ge 0$ and a finitely modeled chain endomorphism of $\ZZ G$-complexes 
 \[
 \varphi_*:{\chains * k G}^{(m)}\to {\chains * k G}^{(m)}
 \]
 extending the identity on $\ZZ$, such that $v(\varphi_q(c))-v(c)\ge K$ for every $c\in \chains q k G$, $q=0,\ldots,m$.
\end{thma}
In this way, a criterion for $\Sigma^m(\cdot,\ZZ)$ is a chain endomorphism that stays in the same index and raises $\chi$-value. Again we need this chain endomorphism to be finitely modeled, which is a property we don't need to impose in \cite[Theorem~B]{Hartmann2024a}. A proof of Theorem~\ref{thma:varphi} can be found in Theorem~\ref{thm:sigmacrit_lc_homol}.

A direct consequence of Theorem~\ref{thma:mu} and Theorem~\ref{thma:varphi} is that computing type $\Clg m$ for the ambient group $G$ is already the bulk of the work when one wants to compute $\TopS^m$.
\begin{thma}
 \label{thma:crit2}
 Let $G$ be a locally compact Hausdorff group of type $\Clg m$. Then there exists $k\ge0$ such that for every nonzero character $\chi:G\to \RR$ the following are equivalent:
 \begin{enumerate}
  \item $\chi\in \TopS^m(G;\ZZ)$;
  \item there exist $K>0$ and a finitely modeled chain endomorphism of $\ZZ G$-complexes 
  \[
  \varphi_*:{\chains * k G}^{(m)}\to {\chains * k G}^{(m)}
  \]
  extending the identity on $\ZZ$ such that for every $q=0,\ldots,m$ and $c\in \chains q k G$ we have
  \[
   v(\varphi_q(c))-v(c)\ge K.
  \]
 \end{enumerate}
\end{thma}
A proof of Theorem~\ref{thma:crit2} can be found in Theorem~\ref{thm:crit2}.

In the classical setting, homological $\Sigma$-invariants form an open subset in the sphere of characters modulo positive multiples. One could also say that they form a cone over an open subset in $\Hom(G,\RR)$. The same holds also true for locally compact $\Sigma$-invariants.
\begin{thma}
\label{thma:open}
 The subset $\TopS^m(G,\ZZ)$ is a cone over an open subset in $\TopHom(G,\RR)$ provided $\Sigma^m\not=\emptyset$.
\end{thma}
A proof of Theorem~\ref{thma:open} can be found in Theorem~\ref{thm:open}.

Suppose that we are faced with a short sequence of locally compact groups $1\to N\to G\to Q\to 1$ that is exact as abstract groups and where $N$ is a closed subspace of $G$ and $Q$ is endowed with the quotient topology from $G$. If $\chi:G\to \RR$ is a character that vanishes on $N$ then it descends to a character $\bar\chi:Q\to \RR$. In case the sequence is split exact, the following result is immediate.

\begin{lema}
\label{lema:semidirect}
  If $G=N\rtimes Q$ is a semidirect product with $N$ closed in $G$ and $\chi$ is a character on $G$ with $\chi\in\TopS^m(G,\ZZ)$ then $\bar\chi\in \TopS^m(Q,\ZZ)$.
\end{lema}
A proof of Lemma~\ref{lema:semidirect} can be found in Lemma~\ref{lem:semidirect}.

Ultimately, in the case of an exact sequence of topological groups that does not necessarily split, one can also use chain endomorphisms as a tool, which in this case are not $G$-equivariant.
\begin{thma}
\label{thma:ge1}
 If $1\to N\to G\to Q\to 1$ is a short exact sequence of locally compact groups then: 
 \begin{enumerate}
  \item if $N$ is of type $\Clg m$ and $\bar\chi\in \TopS^m(Q;\ZZ)$ then $\chi\in 
\TopS^m(G;\ZZ)$;
  \item if $N$ is of type $\Clg {m-1}$ and $\chi\in \TopS^m(G;\ZZ)$ then $\bar\chi\in \TopS^m(Q;\ZZ)$.
 \end{enumerate}
 In particular:
 \begin{enumerate}
  \item if $N$ is of type $\Clg m$ and $Q$ is of type $\Clg m$ then $G$ is of type $\Clg m$;
  \item if $N$ is of type $\Clg {m-1}$ and $G$ is of type $\Clg m$ then $Q$ is of type $\Clg m$.
 \end{enumerate}
\end{thma}
A proof of Theorem~\ref{thma:ge1} can be found in Theorem~\ref{thm:ge1}.

We close this study with an application. Ultimately, the $\Sigma$-theory for locally compact groups helps to establish compactness properties of closed normal subgroups provided that the factor group is abelian.
\begin{thma}
\label{thma:ge2}
  Let $N\trianglelefteq G$ be a closed normal subgroup with abelian factor group $Q=G/N$. Then $N$ is of type $\Clg m$ if $S(G,N)\subseteq \TopS^m(G;\ZZ)$.
\end{thma}
A proof of Theorem~\ref{thma:ge2} can be found in Theorem~\ref{thm:ge2}.

\subsubsection*{Acknowledgements}
Many thanks for helpful conversations go to Ilaria Castellano, Dorian Chanfi, Tobias Hartnick, Stefan Witzel, and Xiaolei Wu. We also thank the anonymous referee for various helpful comments.

This research was funded in part by the DFG grant BU 1224/4-1
 within SPP 2026 Geometry at Infinity.

We also acknowledge the support of the Deutsche Forschungsgemeinschaft (DFG, German Research Foundation) –
Project-ID 491392403 – TRR 358.

%% file: recap.tex
\section{Recap of the homotopical version}
This section gives an overview of the homotopical part of our research on locally compact $\Sigma$-invariants. We start with a few basic definitions from \cite{Hartmann2024c}.

If $T$ is a set, then $\E T$ denotes the free simplicial set on $T$. That is, its $q$-simplices are $\E T_q:=T^{q+1}$ ($q=0,1,\ldots$) where
\begin{align*}
 d_i:\E T_q&\to \E T_{q-1}\\
 (t_0,\ldots,t_q)&\mapsto (t_0,\ldots,\hat t_i,\ldots,t_q)
\end{align*}
stands for the $i$th face map ($i=0,\ldots,q$) and
\begin{align*}
 s_i:\E T_q&\to \E T_{q+1}\\
 (t_0,\ldots,t_q)&\mapsto (t_0,\ldots,t_i,t_i,\ldots,t_q)
\end{align*}
stands for the $i$th degeneracy map ($i=0,\ldots,q$). If $G$ is a locally compact group, then $G$ acts on $\E G$ by left translation on the vertices. Denote
by $\mathcal C$ the collection of compact subsets of $G$. Then $(G\cdot \E C)_{C\in 
\mathcal C}$ forms a filtration of $\E G$. Also, if $\chi:G\to \RR$ is a character, then we define a subset
\[
 G_\chi:=\{g\in G\mid \chi(g)\ge 0\}
\]
of $G$. Then $((G\cdot \E C)\cap \E G_\chi)_{C\in \mathcal C}$ forms a filtration of $\E G_\chi$.

A filtered system $(A_i)_{i\in I}$ of groups is said to be \emph{essentially trivial} if for each $i\in I$ there exists $j\in I$ such that $A_i\to A_j$ is trivial.

\begin{dfn}\cites[Definition 1.3.4]{Abels1997}{Hartmann2024c}
\label{dfn:AT}
 If $m\in \NN$ then
 \begin{itemize}
  \item $G$ is said to be \emph{of type $\Ctp m$} if $\pi_q(G\cdot \E C)_{C\in 
\mathcal C}$ is essentially trivial for $q=0,\ldots,m-1$;
  \item $\TopS^m(G)$ is the collection of characters $\chi:G\to\RR$ 
such that   $\pi_q((G\cdot \E C)\cap \E G_\chi)_{C\in\mathcal C}$ is essentially 
trivial for $q=0,\ldots,m-1$;
 \item $G$ is said to be \emph{of type $\Clg m$} if $\tilde \hlgy_q(G\cdot \E C)_{C\in \mathcal C}$ is essentially trivial for $q=0,\ldots,m-1$.
 \end{itemize}
\end{dfn}
The paper \cite{Hartmann2024c} contains homotopical versions of Theorem~\ref{thma:classical}, Theorem~\ref{thma:mu}, Theorem~\ref{thma:varphi}, Theorem~\ref{thma:ge1} and Theorem~\ref{thma:ge2}. It also contains separate discussions of $\TopS^1$ (for $m=1$ the homotopical $\TopS$ coincides with the homological $\TopS$) and of $\TopS^2$ (one can think of this as compactly presented in a direction). 

%% file: def+basics.tex
\section{Ind-objects and the Vietoris-Rips Complex}
\label{sec:def+basics}
This section discusses the basic notions used in this paper.

If $T$ is a metric space and $k\ge 0$, then a certain collection of $q$-simplices in $\E T$ is given by
\[
 \Delta^q_k:=\{(t_0,\ldots,t_q)\mid d(t_i,t_j)\le k;i,j=0,\ldots,q\}.
\]
The subcomplex $\VR k T$ of $\E T$ with $q$-simplices ${\VR k T}_q:=\Delta^q_k$ is called the \emph{Vietoris--Rips complex for the value $k$}. If we let $k$ vary to $\infty$, then $(\VR k T)_{k\ge 0}$ forms a filtration of $\E T$. We consider the filtered system ${\VR k T}_{k\ge 0}$ as an ind-object in the category of simplicial sets. For the convenience of the reader we recall ind-objects:

If $\mathcal C$ is a category, then $\Ind(\mathcal C)$ describes the category of ind-objects in $\mathcal C$. Objects of this category are functors $X:\mathcal I\to \mathcal C$ where $\mathcal I$ is a small filtered category. We also write $(X_i)_i$ with $X_i:=X(i)$ and $\mathcal I$ is usually a poset. If $(X_i)_i$ and $(Y_j)_j$ are two ind-objects, the set of morphisms between $(X_i)_i$ and $(Y_j)_j$ is given by
\[
 \Hom_{\Ind(\mathcal C)}((X_i)_i,(Y_j)_j)=\varprojlim_i\varinjlim_j \Hom_{\mathcal C}(X_i,Y_j)
\]
\cite[Chapter~8]{Grothendieck1963}. Explicitly, a morphism between ind-objects $X:\mathcal I\to \mathcal C$ and $Y:\mathcal J\to \mathcal C$ is given by a map $\varepsilon:\mathcal I\to \mathcal J$ and a $\mathcal C$-morphism $\varphi_i:X_i\to Y_{\varepsilon(i)}$ for each $i\in \mathcal I$ such that for each $i\to i'\in \mathcal I$ there exists $j\in\mathcal J$ with $\varepsilon(i')\to j\in \mathcal J$ such that the diagram
\[
\xymatrix{
 X_i\ar[d]\ar[r]^{\varphi(i)}
 & Y_{\varepsilon(i)}\ar[dr]\\
 X_{i'}\ar[r]_{\varphi(i')}
 &Y_{\varepsilon(i')}\ar[r]
 &Y_j
 }
\]
commutes. Two such data $(\varepsilon,(\varphi_i)_i),(\delta,(\psi_i)_i)$ define the same morphism if for each $i\in \mathcal I$ there is some $j\in \mathcal J$ with $\varepsilon(i)\to j,\delta(i)\to j\in \mathcal J$ and
\[
\xymatrix{
&Y_{\varepsilon(i)}\ar[rd]&\\
 X_i\ar[ru]^{\varphi_i}\ar[r]_{\psi_i}
 &Y_{\delta(i)}\ar[r]
 &Y_j
 }
\]
commutes \cite{Abels1997}. 

Two filtrations $\cdots \subseteq X_i\subseteq X_{i+1}\subseteq \cdots T$ and $\cdots \subseteq Y_j\subseteq Y_{j+1}\subseteq \cdots T$ of the same object are called \emph{cofinal} if for every $i$ there exists some $j$ with $X_i\subseteq Y_j$ and for every $j$ there exists some $i$ with $Y_j\subseteq X_i$. It is easy to see that two cofinal filtrations are isomorphic as ind-objects. In addition, the property essentially trivial on an ind-object $(A_i)_i$ in the category $\mathrm{Ab}$ of abelian groups is equivalent to being isomorphic as an ind-object to the direct system $(N_i)_{i\in I}$ with $I=\{1\}$ and $N_1=0$. 

\section{Topological groups as coarse objects}

This section starts with a few basic observations that lead to the definition of the homological version of locally compact $\Sigma$-invariants.

If a topological group $G$ is a countable union of compact subsets then it is called \emph{$\sigma$-compact}. In particular, if $G$ is compactly generated, say by $C$, then $G=\bigcup_{n\ge 0}(C\cup C^{-1})^n$ is a countable union of compacts and therefore $\sigma$-compact. Suppose $G$ is compactly generated by $C$ with $C=C^{-1}$ and $1_G\in C$. We define a metric on $G$ in the following way. Say $C^0=\{1_G\}$, then
\[
 d(g,h)=n \mbox{ if }g^{-1}h\in C^n\setminus C^{n-1}.
\]
Note that, in general, this metric does not induce the topology of $G$.

If $C\subseteq G$ is a compact subset of a topological group, define
\[
 \Delta^q_C:=\{(g_0,\ldots,g_q)\in G^{q+1}\mid g_i^{-1}g_j\in C;i,j=0,\ldots,q\}.
\]
As usual $\mathcal C$ denotes the set of compact subsets in $G$.

\begin{lem}
\label{lem:cofinal}
 Let $q\ge 0$. 
 \begin{itemize}
  \item If $G$ is a topological group then the filtrations $(\Delta_C^q)_{C\in\mathcal C}$ and $G(\E C)_{C\in \mathcal C}$ of $G^{q+1}$  are cofinal;
  \item if $G$ is locally compact Hausdorff and compactly generated, say by $\mathcal X$, then the filtrations $(\Delta_C^q)_{C\in \mathcal C}$ and $(\Delta_k^q)_{k\ge 0}$ (where the distance is induced from $\mathcal X$) of $G^{q+1}$ are cofinal.
  \end{itemize}
\end{lem}
\begin{proof}
  Let $C\subseteq G$ be a compact subset, and let $(g_0,\ldots,g_q)\in G(\E C)_q$ 
be a simplex. Then $g_i=gc_i$ for some $g\in G, c_i\in C,i=0,\ldots,q$. Then
$g_i^{-1}g_j=(gc_i)^{-1}gc_j=c_i^{-1}c_j\in C^{-1}C$. Thus, we have shown $G(\E 
C)_q\subseteq \Delta^q_{C^{-1}C}$. Now, let $(g_0,\ldots,g_q)\in \Delta^q_C$ be a
simplex. Then 
 \[
 (g_0,\ldots,g_q)=g_0(1,g_0^{-1}g_1,\ldots,g_0^{-1}g_q)\in G(\E C)_q.
 \]
 Thus $\Delta_C^q\subseteq G(\E C)_q$ This proves the first statement.
 
 Since $G$ is locally compact Hausdorff it is in particular a Baire space. Suppose $\mathcal X=\mathcal X^{-1}$ and $1_G\in\mathcal X$. Then $({\mathcal X}^k)_k$ and $\mathcal C$ are cofinal by a similar reasoning as in \cite[Lemma~3.11]{Hartmann2024c}. This implies that $(\Delta^q_k)_k=(\Delta_{{\mathcal X}^k})_k$ and $(\Delta^q_C)_{C\in \mathcal C}$ are cofinal.
\end{proof}

If $X$ is a set, then a collection $\mathcal E$ of subsets of $X\times X$ forms a \emph{coarse structure} if
\begin{enumerate}
 \item $\Delta_0:=\{(x,x)\mid x\in X\}\in\mathcal E$;
 \item $\mathcal E$ is closed under taking subsets;
 \item $\mathcal E$ is closed under taking finite unions;
 \item for every $x,y\in X$ we have $\{(x,y)\}\in \mathcal E$;
 \item if $E,F\in\mathcal E$ then $E\circ F:=\{(a,c)\mid \exists b\in X: (a,b)\in E, (b,c)\in F\}\in\mathcal E$.
\end{enumerate}
A subset $\mathcal F\subseteq \mathcal E$ of a coarse structure is called a \emph{base of (a coarse structure) $\mathcal E$} if $\mathcal F$ and $\mathcal E$ are cofinal filtrations of $X\times X$. A particular example is the following. If $X$ is a metric space, then $(\Delta_k^1)_k$ forms a base for a coarse structure. We call a set equipped with a coarse structure \emph{coarse space}.

\begin{lem}
 If $G$ is a locally compact Hausdorff group, the filtered system $(\Delta^1_C)_{C\in \mathcal C}$ forms a base for a coarse structure. If in addition $G$ is $\sigma$-compact then there exists a metric on $G$ that induces the coarse structure.
\end{lem}
\begin{proof}
 We check the axioms of a coarse structure. Since points in a Hausdorff space are compact, the set $\{1_G\}$ is compact and $\Delta_0=\Delta^1_{\{1_G\}}$ is an element of the coarse structure. That is $1$ of coarse. If $E\subseteq \Delta^1_C$ and $F\subseteq E$, then of course $F\subseteq \Delta^1_C$. That is 2 of coarse. If $E\subseteq \Delta^1_C$ and $F\subseteq \Delta^1_D$, then $E\cup F\subseteq \Delta^1_{C\cup D}$. That is 3 of coarse. If $g,h\in G$, then $(g,h)\in \Delta^1_{\{g^{-1}h\}}$. That is 4 of coarse. If $E\subseteq \Delta^1_C$ and $F\subseteq \Delta^1_D$, then $E\circ F\subseteq \Delta^1_{CD}$. That is 5 of coarse.
 
 Now suppose $G$ is $\sigma$-compact. Let $(C_i)_{i\in \NN}$ be compact sets in
$G$ with $\bigcup_i C_i=G$. If $D\in\mathcal C$ is another compact set then
since $G$ is a Baire space there exists some $i$ with $D\subseteq C_i$. So 
$(C_i)_{i\in\NN}$ and $\mathcal C$ are cofinal. This implies that
$(\Delta^1_{C_i})_{i\in \NN}$ is also a base for the coarse structure induced by
$(\Delta^1_C)_{C\in\mathcal C}$. This base contains countably many elements. By \cite[Theorem~2.55]{Roe2003}, this implies that there exists a metric on $G$ that induces this coarse structure.
\end{proof}

A (not necessarily continuous) mapping $\varphi:X\to Y$ between metric spaces is called
\begin{itemize}
 \item \emph{coarsely Lipschitz} if for every $R\ge 0$ there exists $S\ge 0$ with $\varphi^{\times 2}(\Delta^1_R)\subseteq \Delta^1_S$;
 \item \emph{close to} another mapping $\psi:X\to Y$ if there exists $H\ge0$ with $(\varphi\times \psi)(\Delta_0^1)\subseteq \Delta^1_H$.
\end{itemize}
Coarse spaces and coarsely Lipschitz mappings modulo close form a category called the \emph{coarse category}.

\begin{lem}
 If $\alpha:G\to H$ is a continuous group homomorphism, then it is coarsely Lipschitz. In particular, if $\alpha$ is a group isomorphism and a homeomorphism, then $\alpha$ is an isomorphism in the coarse category.
\end{lem}
\begin{proof}
  Let $C\subseteq G$ be compact. We show $\alpha^{\times2}(\Delta^1_C)\subseteq \Delta^1_{\alpha(C)}$. If $(g,h)\in \Delta^1_C$, then $g^{-1}h\in C$, which implies $\alpha(g)^{-1}\alpha(h)=\alpha(g^{-1}h)\in \alpha(C)$. Thus, $\alpha$ is coarsely Lipschitz. This proves the first claim. If $\alpha$ is additionally a group isomorphism and a homeomorphism, then the set-theoretic inverse $\alpha^{-1}$ is also continuous and a group homomorphism which implies $\alpha^{-1}$ is coarsely Lipschitz. Also, the compositions $\alpha\circ\alpha^{-1}$ and $\alpha^{-1}\circ \alpha$ are the identities, in particular close to the identities, which proves the second claim.
\end{proof}
 It is well-known that countable groups have bounded geometry. We now discuss the locally compact setting. For that, we recall the definition of bounded geometry \cite{Roe2003}: If $E$ is an entourage (that is, an element of the coarse structure) in a metric space $X$ and $S\subseteq X$ a subset, then the $E$-capacity $\capacity_E(S)$ of $S$ is the largest number $m$ for which there exist points $y_1,\ldots,y_m\in S$ such that no pair $(y_i,y_j)$ of distinct points belongs to $E$. Then $X$ is said to have \emph{bounded geometry} if there exists an entourage $E$ (which is called the gauge) such that for every entourage $F$ the supremum
 \[
  \capacity_E(F):=\sup_{x\in X}\max(\capacity_E(F[E[x]]),\capacity_E(F^{-1}[E[x]]))
 \]
is finite. We explain the notation $E[C]$: If $E\subseteq X\times X$ and $C\subseteq X$ are subsets, where $X$ is a metric space (usually $E$ is an entourage and $C$ is a bounded set), then 
\[
E[C]:=\{x\in X\mid \exists y\in C\mbox{ s.t. }(x,y)\in E\}.
\]

\begin{prop}
 If $G$ is a locally compact group, then $G$ has bounded geometry.
\end{prop}
Compare with \cite[Lemma~5.3]{Rosendal2022}.
\begin{proof}
 Let $C=C^{-1}$ be a compact neighborhood of $1_G$. We show $G$ has bounded geometry with gauge $\Delta_C$. If $D\subseteq G$ is compact then
 \[
  \Delta_D[\Delta_C[x]]=\Delta_D[xC]=xCD.
 \]
Then 
\[
 \capacity_{\Delta_C}(\Delta_D)=\sup_{x\in G}\capacity_{\Delta_C}(\Delta_D[\Delta_C[x]])=\sup_{x\in G}\capacity_{\Delta_C}(xCD)=\capacity_{\Delta_C}(CD).
\]
If $(y_i)_i$ are points in $CD$ with $(y_i,y_j)\not\in \Delta_{\mathring C}$ for $i\not=j$ maximal then $y_i^{-1}y_j\not \in \mathring C$. Thus $y_j\not\in y_i \mathring C$. By maximality $\bigcup_i y_i \mathring C$ form an open cover of $CD$. Since $CD$ is compact there exists a finite subcover. Since the $y_i$ are not covered by other open subsets $y_z\mathring C$, the index set $i$ needs to be finite. Thus
\[
 \capacity_{\Delta_C}(CD)\le \capacity_{\Delta_{\mathring 
C}}(CD)<\infty.\qedhere
\]
\end{proof}

A metric space $X$ is said to be \emph{uniformly locally finite} if for every entourage $E\subseteq X\times X$ there exists $N\ge 0$ such that every $E$-ball contains at most $N$ points. That is, 
\[
 |E[x]|\le N.
\]

If $G$ has bounded geometry with gauge $\Delta_C$ then $X\subseteq G$ is defined to be a maximal $\Delta_C$-separated subset. Then $\bigcup_{x\in X}xC=G$ by maximality. Thus, the inclusion $X\to G$ is coarsely surjective. In addition, $X$ is uniformly locally finite, since for every compact $D\subseteq G$ the number of points in $\Delta_D[x]\cap X$ is bounded by $\capacity_{\Delta_C}(\Delta_D)$. In this way, we found a metric space $X$ that is uniformly locally finite and has the same coarse type as $G$.

\section{Definition and basic properties}

This section contains the definition of $\TopS^m(\cdot,\ZZ)$ and a proof of Proposition~\ref{propa:veryeasy} and Lemma~\ref{lema:semidirect}.

Simplicial homology will play an important role in this paper. If $k\ge 0$ and $G$ is a group, then $(\chain_q(\VR k G),\partial_q)$ denotes the simplicial chain complex assigned to the simplicial set $\VR k G$. That is,
\[
 \chains q k G:=\ZZ[\Delta_k^q]
\]
and
\begin{align*}
 \diff q:\chains q k G&\to \chains {q-1} k G\\
 (g_0,\ldots,g_q)&\mapsto \sum_{i=0}^q(-1)^i(g_0,\ldots,\hat g_i,\ldots,g_q).
\end{align*}
The group $G$ acts on $\ZZ[\Delta_k^q]$ by left translation on the vertices, and $\partial_q$ is $G$-equivariant. In this way, $(\chains q k G,\partial_q)$ is also a chain complex of free $G$-modules. If $k$ varies to $\infty$ they form a filtration of the free resolution $(\ZZ[G^{q+1}],\partial_q)$ of the constant $G$-module $\ZZ$.

\begin{dfn}
\label{dfn:main}
 Let $G$ be a locally compact, Hausdorff, $\sigma$-compact group, and $R$ a commutative ring. Then $\chi\in\TopS^1(G;R)$ if $\hlgy_0((R[\Delta^q_c\cap (G_\chi)^{q+1}],\partial_q))=R$ for some $c\ge 0$. If $m\ge 2$ then $\chi\in\TopS^m(G;R)$ if $\chi\in \TopS^1(G;R)$ and $\hlgy_q((R[\Delta^q_k\cap (G_\chi)^{q+1}],\partial_q))_{k\ge 0}$ is essentially trivial for $q=1,\ldots,m-1$.
\end{dfn}

If $X$ is a metric space and $c\ge 0$ then a finite sequence $a_0,\ldots, a_n\in X$ of points defines a \emph{$c$-path} if $d(a_i,a_{i+1})\le c$ for every $i=0,\ldots,n-1$. The space $X$ is \emph{$c$-coarsely connected} if every two points can be joined by a $c$-path. It is \emph{coarsely connected} if it is $c$-coarsely connected for some $c\ge 0$\cite{Cornulier2016}.

\begin{dfn}
\label{dfn:finiteprops-metric2}
A metric space $X$ is said to be of type $\Flg 1$ if it is coarsely connected. It is said to be of type $\Flg m$, $m\ge 2$ if it is of type $\Flg 1$ and for every $1\le q \le m-1$ the directed set of groups $(\homology q r X)_r$ is essentially trivial. Equivalently, a metric space $X$ is of type $\Flg m$ if the directed set of reduced simplicial homology groups $\tilde \hlgy_q(\VR rX)_r$ is essentially trivial for $0\le q\le m-1$.

The space $X$ is said to be of type $F_m$ if $\pi_q(\VR r X)_r$ is essentially trivial for $q=0,\ldots,m-1$.
\end{dfn}

\begin{prop}
\label{prop:homotopic_metric}
If $G$ is a locally compact, Hausdorff, $\sigma$-compact group, then $\chi\in \TopS^m(G)$ if and only if there is a centered compact $C\subseteq G$ containing $1_G$ such that $\pi_q(\Delta^*_{C^k}\cap G_\chi^{*+1})_k$ is essentially trivial for $q=0,\ldots,m-1$.
\end{prop}
\begin{proof}
Recall that $\chi\in \TopS^m(G)$ if $\pi_q(G_\chi \E C)_C$ is essentially trivial for $q=0,\ldots,m-1$. By \cite[Lemma~3.7]{Hartmann2024c} this is equivalent to saying that there is a centered compact $C$ containing $1_G$ with $\pi_q((G_\chi \cdot \E C^m )\cap \E G_\chi)_m$ essentially trivial for $q=0,\ldots,m-1$.
  
It remains to show that $((G_\chi \cdot \E C^m )\cap \E G_\chi)_m$ is cofinal to $((\Delta^q_{C^m}\cap G_\chi^{q+1})_q)_m$. By Lemma~\ref{lem:cofinal} the filtration $((\Delta^q_{C^m}\cap G_\chi^{q+1})_q)_m$ is cofinal to $(G \E C^m\cap \E G_\chi)_m$. We are left to show that $(G \E C^m\cap \E G_\chi)_m$ is cofinal to $((G_\chi \cdot \E C^m )\cap \E G_\chi)_m$.
  
Clearly 
\[
(G_\chi \E C^m)\cap \E G_\chi\subseteq G \E C^m\cap \E G_\chi.
\]
For the other direction, suppose $(g_0,\ldots,g_q)\in G EC^m\cap EG_\chi$. Then $g_i=gc_i$ with $g\in G,c_i\in C$ and $\chi(gc_i)\ge 0$. Then $gc_0\cdot (c_0^{-1}c_0,\ldots,c_0^{-1}c_q)\in G_\chi EC^{2m}\cap EG_\chi$. This proves $(G\E C^m)\cap \E G_\chi\subseteq (G_\chi \E C^{2m})\cap \E G_\chi$ and, therefore, the last claim. 
\end{proof}

\begin{cor}
\label{cor:supermetric}
Let $G$ be a locally compact Hausdorff $\sigma$-compact group and $m\ge 1$. Then
\begin{enumerate}
\item $G$ is of type $\Clg m$ if and only if $(G,d)$ is of type $\Flg m$;
\item $G$ is of type $\Ctp m$ if and only if $(G,d)$ is of type $\Ftp m$;
\item if $G$ is of type $\Ctp m$ then: $\chi\in \TopS^m(G)$ if and only if $(G_\chi,d|_{G_\chi})$ is of type $\Ftp m$;
\item if $G$ is of type $\Clg m$ then: $\chi\in \TopS^m(G;\ZZ)$ if and only if $(G_\chi,d|_{G_\chi})$ is of type $\Flg m$.
\end{enumerate}
\end{cor}
\begin{proof}
If $G$ is not compactly generated, then $(G,d)$ is not coarsely connected, so $(G,d)$ is neither of type $\Flg m$ nor of type $\Ftp m$. At the same time every group of type $\Clg m/\Ctp m$ is compactly generated. This proves the claims in the case that $G$ is not compactly generated. Now suppose $G$ is compactly generated, say by $\mathcal X$, and $d$ is the word-length metric according to $\mathcal X$. 

We begin with the first claim. By Definiton~\ref{dfn:AT} the group $G$ is of type $\Clg m$ if $\tilde \hlgy_q(G\cdot \E C)_{C\in \mathcal C}$ is essentially trivial for $q=0,\ldots,m-1$. By Lemma~\ref{lem:cofinal} this is equivalent to $\tilde \hlgy_q(\Delta_k^q)_{k\ge 0}$ is essentially trivial for $q=0,\ldots,m-1$. And that is type $\Flg m$ of Definition~\ref{dfn:finiteprops-metric2}. That is the first claim.

Now we prove the second claim. By Definition~\ref{dfn:AT} the group $G$ is of type $\Ctp m$ if $\pi_q(G\cdot \E C)_{C\in \mathcal C}$ is essentially trivial for $q=0,\ldots,m-1$. By Lemma~\ref{lem:cofinal} this is equivalent to $\pi_q(\Delta_k^q)_{k\ge 0}$ is essentially trivial for $q=0,\ldots,m-1$. And that is type $\Ftp m$ of Definition~\ref{dfn:finiteprops-metric2}. That is the second claim.

The third claim is Proposition~\ref{prop:homotopic_metric} and the fourth claim is Definition~\ref{dfn:main}.
\end{proof}

\begin{prop}
\label{prop:veryeasy}
If $G$ is a locally compact, Hausdorff group and $m\ge1$ then the following are equivalent:
\begin{enumerate}
 \item $0\in \TopS^m(G;\ZZ)$;
 \item $\TopS^m(G;\ZZ)\not=\emptyset$;
 \item $G$ is of type $\Clg m$.
\end{enumerate}
\end{prop}
\begin{proof}
It is obvious that condition 1 implies condition 2.

Now we show that condition 2 implies condition 3. Suppose condition 2, that $\TopS^m(G;\ZZ)\not=\emptyset$.  First, we discuss the $m=1$ case. Suppose $\chi\in \TopS^1(G,\ZZ)$. Then $\chi\in \TopS^1(G)$ by Lemma~\ref{lem:sigma1}. Thus, $G$ has type $\Ctp 1$ by \cite[Proposition~3.5]{Hartmann2024c}. This is equivalent to having type $\Clg 1$ \cite{Abels1997}. Now we prove the case $m\ge 2$. Suppose $\chi\in \TopS^m(G,\ZZ)$ and we have already shown that $G$ is of type $\mathrm{CP}_{m-1}$. Then $(G_\chi,d)$ has type $\Flg m$. Suppose $\hlgy_{m-1}(\ZZ[\Delta_C^{m-1}\cap G_\chi^m])$ vanishes in $\hlgy_{m-1}(\ZZ[\Delta_D^{m-1}\cap G_\chi^m])$. Let $z\in \ZZ[\Delta^{m-1}_C]$ be a cycle. If $\chi=0$, then $G_\chi=G$ and we are done. Otherwise, there exists some $t\in G$ with $\chi(t)>0$. Then $t^nz\in\ZZ[\Delta^{m-1}_C\cap (G_\chi)^{m}]$ for some $n\ge 0$. Then there exists some $c\in \ZZ[\Delta^{m}_D\cap (G_\chi)^{m+1}]$ with $\partial_{m}c=t^nz$. Then $t^{-n}c\in \ZZ[\Delta^{m}_D]$ with $\partial_{m}t^{-n}c=z$. Thus, $\hlgy_{m-1}(\ZZ[\Delta_C^{m-1}])$ vanishes in $\hlgy_{m-1}(\ZZ[\Delta_D^{m-1}])$. This implies that the group $G$ is of type $\Clg m$. This is condition 3.

Now we show that condition 3 implies condition 1. Suppose $G$ is of type $\Clg m$ and let $\chi=0$ be the zero character on $G$. Then $(G_\chi,d)=(G,d)$ is of type $\Flg m$ which proves $0\in \TopS^m(G;\ZZ)$, condition 1.
 \end{proof}

A \emph{uniform lattice} in a locally compact Hausdorff $
\sigma$-compact group $G$ is a cocompact discrete subgroup $\Gamma\le G$ \cite{Cornulier2016}.

\begin{lem}
 If $\Gamma\le G$ is a uniform lattice then 
 \begin{itemize}
  \item $\chi\in \TopS^k(G,\ZZ)$ if and only if $\chi|_\Gamma\in \Sigma^k(\Gamma,\ZZ)$;
  \item $\chi\in\TopS^k(G)$ if and only if $\chi|_\Gamma\in\Sigma^k(\Gamma)$. 
 \end{itemize}
\end{lem}
\begin{proof}
 We show that $G_\chi$ and $\Gamma_{\chi|_\Gamma}$ are isomorphic in the coarse category, which proves the claim.
 
 There is an inclusion $\iota:\Gamma_{\chi|_{\Gamma}}\subseteq G_\chi$ since $g\in \Gamma_{\chi|_\Gamma}$ implies $0\le\chi|_\Gamma(g)=\chi(g)$. Suppose $K$ is compact with $K=K^{-1}$ that contains $1_G$ and $\Gamma K=G$. Since $\chi$ is continuous and $K$ compact, the image $\chi(K)$ is bounded. If $\chi$ vanishes on $\Gamma$, then it also vanishes on $G$. So, suppose $\chi|_\Gamma$ does not vanish. Then there exists some $t\in \Gamma$ with $\chi(t)\ge \max_{g\in K} \chi(g)$. And $K\le C^n$ for some $n$ where $C\subseteq G$ is the compact generating set. Now for every $g\in G$ there exists $\varphi(g)\in \Gamma$ with $\varphi(g)^{-1}g\in K$. Without loss of generality $\varphi(g)=g$ if $g\in \Gamma$. Then define a map
 \begin{align*}
  \alpha:G_\chi&\to \Gamma_{\chi|_\Gamma}\\
  g&\mapsto \varphi(g)t.
 \end{align*}
 This map is coarsely Lipschitz since
 \begin{align*}
  d(\varphi(g)t,\varphi(h)t)
  &\le d(\varphi(g)t,\varphi(g))+d(\varphi(g),g)+d(g,h)+d(h,\varphi(h))+d(\varphi(h),\varphi(h)t)\\
  &\le \ell(t)+n+d(g,h)+n+\ell(t).
 \end{align*}
 Moreover $\alpha$ is well-defined since
 \[
  \chi|_\Gamma(\varphi(g)t)=\chi(\varphi(g)t)=-\chi(\varphi(g)^{-1}g)+\chi(g)+\chi(t)\ge 0.
 \]
 Now 
 \[
  d(\alpha\circ\iota(g),g)=d(g,gt)\le \ell(t).
 \]
 Thus $\alpha\circ\iota$ is close to the identity. Likewise
 \[
  d(\iota\circ\alpha(g),g)=d(\varphi(g)t,g)\le d(\varphi(g)t,\varphi(g))+d(\varphi(g),g)\le \ell(t)+n.
 \]
This implies $\iota\circ \alpha $ is close to the identity. Thus we show that $\alpha,\iota$ are coarse inverses.
\end{proof}

As usual, if $T\subseteq G$ is a subset, then
\[
 S(G,T):=\{\chi:G\to \RR\mid \chi|_T=0\}
\]
denotes the set of characters on $G$ that vanish on $T$.
 \begin{prop}
 \label{prop:center}
  If $m\ge 1$ and $G$ is a group of type $\Clg m$ with center $Z$ then \[
   S(G,Z)^c\subseteq \TopS^m(G;\ZZ).
  \]
 \end{prop}
 \begin{rem}
     In the discrete case, this has been done in \cite[Lemma~2.1]{Meier2001}.
 \end{rem}
\begin{proof}
 Let $\chi$ be a character that does not vanish on the center. Then there exists some $t\in Z$ with $\chi(t)<0$. Without loss of generality, $t$ is an element of the generating compact which induces the metric on $G$. The mapping
 \begin{align*}
  h_q:\ZZ[\Delta^q_k\cap G_\chi^{q+1}]&\to \ZZ[\Delta^{q+1}_{k+1}]\\
  (g_0,\ldots,g_q)&\mapsto \sum_{i=0}^q(-1)^i(g_0,\ldots,g_i,tg_i,\ldots,tg_q)
 \end{align*}
is a chain homotopy between the inclusion $\chains q k {G_\chi} \subseteq \chains q {k+1} G$ and the chain homomorphism $c\mapsto tc$ (this map is $G$-equivariant, since $t$ is in the center). Here we use the fact that $t$ is in the center.

Let $z\in \chains q k {G_\chi}$ be a cycle. Then $tz-z$ is a boundary in $\chains q {k+1}G$. Say $c_0\in \chains {q+1}{k+1}G$ with $\diff {q+1}c_0=tz-z$. For each $n\ge 0$ define
\[
 z_n=t^nz, \quad c_n=t^nc_0.
\]
Then
\[
 \diff {q+1} c_n=t^n\diff {q+1}c_0=t^n(tz-z)=z_{n+1}-z_n.
\]
Since $G$ is of type $\Clg {q+1}$ there exists $l$ such that $\homology q k G$ vanishes in $\homology q l G$. So, there exists some $c\in \chains {q+1} l G$ with $\diff {q+1} c=z$. Then $-(n+1)\chi(t)\ge -\min\{\chi(g)\mid g\in\supp c\}=:-v(c)$ for some $n\in \NN$. Define 
\[
 \tilde c=t^{-(n+1)}(c+\sum_{i=0}^nc_i).
\]
Then
\begin{align*}
 \diff{q+1}\tilde c=t^{-(n+1)}(z+\sum_{i=0}^nz_{i+1}-z_i)=t^{-(n+1)}z_{n+1}=z
\end{align*}
and
\begin{align*}
 v(\tilde c)=-(n+1)\chi(t)+\min_i(v(c),v(c_i))=\min(-(n+1)\chi(t)+v(c),-\chi(t)+v(c_0))\ge 0.
\end{align*}
Also, $c\in \ZZ[\Delta^{q+1}_l]$ and $c_i=t^ic_0\in \ZZ[\Delta^{q+1}_{k+1}]$. Thus, $\tilde c\in \ZZ[\Delta^{q+1}_{\max(l,k+1)}]$. In this way, we have shown that $\homology q k {G_\chi}$ vanishes in $\homology q{\max(l,k+1)}{G_\chi}$. Hence $(G_\chi,d)$ has type $\Flg m$ and so by Corollary~\ref{cor:supermetric} we conclude $\chi\in\TopS^m(G;\ZZ)$.
\end{proof}

\begin{lem}
 If $N\le G$ is a normal subgroup, then it is closed if and only if $Q:=G/N$ is Hausdorff. In the case where $G$ is locally compact Hausdorff, both the subspace topology on $N$ and the quotient topology on $Q$ are locally compact Hausdorff.
\end{lem}
\begin{proof}
 Suppose $N$ is closed in $G$. If $g_1N,g_2N\in Q$ are two points with $g_1N\not=g_2N$, then $g_1Ng_2^{-1}$ does not contain $1_G$.So $(g_1Ng_2^{-1})^c$ is an open subset containing $1_G$. Since $G$ is a topological group, there exists a symmetric open $V$ that contains $1_G$ with $VV\subseteq (g_1Ng_2^{-1})^c$. Then $Vg_1N,Vg_2N$ do not intersect: For if $g=v_1g_1n_1=v_2g_2n_2$ with $v_1,v_2\in V,n_1,n_2\in N$ then
 \[
  g_1Ng_2^{-1}\ni g_1n_1n_2^{-1}g_2^{-1}=v_1^{-1}v_2\in VV,
 \]
a contradiction. Thus, $Vg_1N$ and $Vg_2N$ are neighborhoods of $g_1N,g_2N$ in $Q$ that do not intersect. This implies $Q$ is Hausdorff.

If conversely $Q$ is Hausdorff, then the points in $Q$ are closed. This in particular implies that $1_Q$ is closed. Since $Q$ inherits the quotient topology from $G$, the quotient map $\pi:G\to Q$ is continuous. Then $N=\pi^{-1}(1_Q)$ as an inverse image of a closed set under a continuous map is closed.

If additionally $G$ is locally compact Hausdorff, then there exists a compact neighborhood $K$ of $1_G$. Then $\pi(K)$ is a compact neighborhood of $1_Q$. This shows that $Q$ is locally compact. In addition, $K\cap N$ is a compact neighborhood of $1_N$ since $N$ is closed. This shows that $N$ is locally compact. Also, if $x_1,x_2\in N$ are distinct then they are separated in $G$ by neighborhoods $U_1$ and $U_2$. Then $U_1\cap N$ and $U_2\cap N$ are neighborhoods in $N$ that separate $x_1,x_2$ in $N$. Thus, $N$ is also Hausdorff.
\end{proof}

\begin{lem}
\label{lem:semidirect}
 If $G=N\rtimes Q$ is a semidirect product with continuous split, $N$ closed in $G$, $\pi:G\to Q$ the projection and $\chi$ is a character on $Q$ with $\chi\circ \pi\in\TopS^m(G,\ZZ)$ then $\chi\in \TopS^m(Q,\ZZ)$.
\end{lem}
Compare with \cites[Theorem~8]{Alonso1994}[Proposition~2.8]{Almeida2018}[Corollary~2.8]{Meinert1997}[Corollary~3.12]{Meinert1996}.
\begin{proof}
 Suppose $\chi\circ\pi\in\TopS^m(G,\ZZ)$. Since $g\in G_{\chi\circ\pi}$ is equivalent to $\chi\circ\pi(g)\ge 0$ which is equivalent to $\pi(g)\in Q_\chi$, we obtain $\pi(G_{\chi\circ\pi})\subseteq Q_\chi$. Also, if $g,g'\in G$ with $\pi(g)=\pi(g')$, then $\chi\circ\pi(g)=\chi\circ\pi(g')$. So, any two elements in the same fiber of $\pi$ have the same $\chi\circ\pi$-value. So $\pi':=\pi|_{G_{\chi\circ\pi}}:G_{\chi\circ\pi}\to Q_\chi$ is well defined.
 
 Let $t:Q\to G$ be a section. This means $\pi\circ t=\id_Q$. If $q\in Q_\chi$, then $\pi\circ t(q)=q\in Q_\chi$, which implies $t(q)\in \pi^{-1}(Q_\chi)=G_{\chi\circ \pi}$. So, the composition of $\pi'$ with $t':=t|_{Q_\chi}$ is well-defined and $\pi'\circ t'=\id_{Q_\chi}$. By assumption $(G_{\chi\circ\pi},d)$ is of type $\Flg m$. Then the homology groups of $(\VR k{G_{\chi\circ\pi}})_k$ as an ind-object vanish in dimensions $q=0,\ldots,m-1$. Now $\id_{Q_\chi}$ factors over $G_{\chi\circ\pi}$ via $\pi',t'$, both of which are coarsely Lipschitz and therefore induce morphisms of ind-objects. This implies that the homology groups of $(\VR k{Q_\chi})_k$ also vanish in dimensions $q=0,\ldots,m-1$. This means $(Q_\chi,d)$ is of type $\Flg m$ and therefore $\chi\in \TopS^m(Q,\ZZ)$.
\end{proof}

\section{Hurewicz-like Theorem}
This section proves Theorem~\ref{thma:hurewicz}.

\begin{lem}
\label{lem:sigma1}
 If $G$ is a locally compact Hausdorff group of type $\Clg 1$, then
$\TopS^1(G;\ZZ)=\TopS^1(G)$.
\end{lem}
\begin{proof}
 Let $\chi$ be a character on $G$. Then $\chi\in \TopS^1(G)$ is
equivalent to $G_\chi$ being $1$-coarsely connected for some compact generating
set $D\subseteq G$ by \cite[Corollary~5.4]{Hartmann2024c}. Suppose that this is the
case. If $x_0,x_1$ are two points in $G_\chi$, then there exists a $1$-path 
$a_0,\ldots,a_n$ in $G_\chi$ that joins $x_0$ to $x_1$. Then
$\partial_1(\sum_{i=0}^{n-1}(a_i,a_{i+1}))=x_1-x_0$. This proves
$\hlgy_0((\ZZ[\Delta^q_D\cap (G_\chi)^{q+1}],\partial_q))=\ZZ$. If, on the other hand, $\hlgy_0((\ZZ[\Delta^q_D\cap (G_\chi)^{q+1}],\partial_q))=\ZZ$ then for every $x_0,x_1\in G_\chi$ the chain $x_1-x_0$ is an element of the image of
$\partial_1$. Say $\partial_1(\sum_{i=0}^mn_i(a_i,b_i))=x_1-x_0$. If $x\in 
G_\chi$ with $x\not=x_0,x_1$ then $\sum_{a_i=x}n_i=\sum_{b_i=x}n_i$. And 
$\sum_{a_i=x_0}n_i=\sum_{b_i=x_0}n_i+1$ and 
$\sum_{a_i=x_1}n_i+1=\sum_{b_i=x_1}n_i$. This means $(a_i)_i,x_1$ aligned in the 
right order describes a $1$-path (with respect to the metric induced by $D$) 
that joins $x_0$ to $x_1$. This implies that $G_\chi$ is $1$-coarsely connected 
with respect to the metric induced by $D$.
\end{proof}

\begin{thm}
  \label{thm:hurewicz}
  If $m\ge 2$ and $X$ is a metric space then it is of type $\Ftp m$ exactly when it is of type $\Ftp 2$ and of type $\Flg m$.
\end{thm}
 \begin{proof}
 We first prove that $\Ftp m$ implies $\Flg m$. If $X$ is of type $\Ftp m$, then
$\pi_q(\VR k X)_k$ is essentially trivial for $q=0,\ldots,m-1$. We use
\cite[Lemma~1.1.3]{Abels1997} which states that there exists a sequence $(B_k)_k$ of
$(m-1)$-connected spaces such that $(\VR k X)_k$ and $(B_k)_k$ are isomorphic as ind-spaces. Then $\tilde \hlgy_q(B_k)=0$ for every $k,q=0,\ldots,m-1$ by the 
Hurewicz theorem. Thus, $\tilde \hlgy_q(\VR k X)$ is essentially trivial. This shows that $X$ is of type $\Flg m$.
 
Now we prove that $\Ftp 2$ and $\Flg m$ implies $\Ftp m$. For that it is sufficient
to prove that for any ind-space $(A_k)_k$ the conditions $\tilde \hlgy_q(A_k)_k$ 
essentially trivial for $q=0,\ldots,m-1$ and $\pi_0(A_k)_k,\pi_1(A_k)_k$ 
essentially trivial imply $\pi_q(A_k)_k$ essentially trivial for
$q=0,\ldots,m-1$. We do that by induction on $q$ starting with $q=1$. Since
$\pi_0(A_k)_k,\pi_1(A_k)_k$ are essentially trivial,
\cite[Lemma~1.1.3]{Abels1997} provides us with a sequence $(B^1_k)_k$ of
$1$-connected spaces that is isomorphic to $(A_k)_k$ as ind-spaces. Now we
consider $m-1\ge q\ge 2$. We can assume there exists a sequence $(B^{q-1}_k)_k$ 
of $(q-1)$-connected spaces which is isomorphic to $(A_k)_k$ as ind-spaces. 
Then by the Hurewicz theorem, the Hurewicz homomorphism $h_q:\pi_q(B_k^{q-1})\to 
\hlgy_q(B_k^{q-1})$ is an isomorphism for every $k$. Since
$\hlgy_q(B^{q-1}_k)_k$ is essentially trivial, there is for every $k$ some $l$ 
such that the left vertical map in
\[
\xymatrix{
 \hlgy_q(B^{q-1}_k;\ZZ)\ar[r]^{h_q^{-1}}\ar[d]_0
 & \pi_q(B^{q-1}_k)\ar[d]\\
 \hlgy_q(B^{q-1}_l;\ZZ)\ar[r]^{h_q^{-1}}
 & \pi_q(B^{q-1}_l)
 }
\]
is zero. Then, since $h_q^{-1}$ is surjective, the right vertical map must be $0$. Thus, $(B^{q-1}_k)_k$ is essentially $q$-connected. By \cite[Lemma~1.1.3]{Abels1997} there is a sequence $(B^q_k)_k$ of $q$-connected spaces that is isomorphic to $(B^{q-1}_k)_k$ as ind-spaces. Since $(B^{q-1}_k)_k$ is isomorphic to $(A_k)_k$, this proves the induction hypothesis for the next step.
 \end{proof}
 
\begin{cor}
\label{cor:hurewicz}
 If $G$ is a locally compact Hausdorff compactly generated group and $m\ge 2$, then
 \[
  \TopS^m(G)=\TopS^2(G)\cap \TopS^m(G;\ZZ).
 \]
\end{cor}
\begin{proof}
 Suppose $G$ is endowed with the word-length metric $d$ of a compact generating set. Then $G_\chi$ inherits this metric from $G$. The claim is then the result of Corollary~\ref{cor:supermetric} and Theorem~\ref{thm:hurewicz}.
\end{proof}

%% file: classical2.tex
\section{Relation to classical Sigma-invariants}
This section proves Theorem~\ref{thma:classical}.

An abstract group $G$ is said to be \emph{of type} $\Flg m$ if there exists a projective resolution
 \[
  \to P_m\to P_{m-1}\to \cdots P_0\to \mathbb Z\to 0 
 \]
 of the constant $G$-module $\mathbb Z$ with $P_m,\ldots,P_0$ finitely generated.
 
 If $\chi:G\to \RR$ is a character, then $G_\chi$ is a monoid. Then $\ZZ[G_\chi]$ is also a ring and we can talk about $G_\chi$-modules: $\chi$ is said to \emph{belong to $\Sigma^m(G,\ZZ)$} if there exists a projective resolution
 \[
  \to P_m\to P_{m-1}\to \cdots P_0\to \mathbb Z\to 0 
 \]
 of the constant $G_\chi$-module $\mathbb Z$ with $P_m,\ldots,P_0$ finitely generated.

\begin{thm}(\cite[Lemma~7, Theorem~(Brown)]{Alonso1994})
\label{thm:alonso}
If $m\ge 1$ and $G$ is a countable group then $G$ is of type $\Flg m$ if and only if for every $0\le q\le m-1$ the directed set of reduced simplicial homology groups $(\tilde \hlgy_q(\VR r X))_r$ is essentially trivial.
\end{thm}
Note that the proof is a particular case of Brown's criterion \cite{Brown1987}.

The right side of Theorem~\ref{thm:alonso} is a quasi-isometric invariant \cite[Corollary~9]{Alonso1994}, and indeed even a coarse invariant. This motivated Definition~\ref{dfn:finiteprops-metric2}.

If a group $G$ is finitely generated by $X$ with $X=X^{-1}$ and $1_G\in X$ then the left-invariant metric on $G$ can be described as
\[
 d(g,h)=n \mbox{ if }g^{-1}h\in X^n\setminus X^{n-1}
\]
where we say $X^0=\{1_G\}$. Note that the metric $d$ on $G$ depends on the choice of generating set $X$, but its geometry at infinity does not. Every subset of $G$ is equipped with the subspace metric from $G$.

\begin{lem}
\label{lem:fgG-module}
 The $G$-module $\ZZ[G^{q+1}]$ is free with basis $(1_G,g_1,\ldots,g_q)$ for $g_1,\ldots, g_q\in G$. The complex $(\ZZ[G^{q+1}],\diff q)$ is a free resolution of the trivial $G$-module $\ZZ$. For each $k\ge0$ the submodule $\ZZ[\Delta_k^q]\le \ZZ[G^{q+1}]$ is also freely generated by the finite set $(1_G,g_1,\ldots,g_q)$ with $d(1_G,g_i)\le k, d(g_i,g_j)\le k$ for every $i,j=1,\ldots,q$. The $(\ZZ[\Delta^q_k])_{k\in \NN}$ form a filtration of $\ZZ[G^{q+1}]$ and $(\ZZ[\Delta^q_k],\diff q)$ form a chain complex. 
\end{lem}
Note that $\ZZ[\Delta^q_k]$ depends on the metric that depends on the choice of generating set~$X$.
\begin{proof}
Since $(1_G,g_1,\ldots,g_q)$ are representatives of the free action of $G$ on $G^{q+1}$ and ultimatively of $\ZZ[G]$ on $\ZZ[G^{q+1}]$ they form a basis. It is standard to check $\partial^2=0$. And $h_q:(g_0,\ldots,g_q)\mapsto (1_G,g_0,\ldots,g_q)$ is a null-homotopy of the identity, so the sequence is exact.

If $(g_0,\ldots,g_q)\in\Delta^q_k$ and $g\in G$, then $d(gg_i,gg_j)=d(g_i,g_j)\le k$ for each $i,j=0,\ldots,q$, which implies $g(g_0,\ldots,g_q)=(gg_0,\ldots,gg_q)\in \Delta^q_k$. So $\ZZ[\Delta^q_k]$ is indeed a $\ZZ[G]$-submodule of $\ZZ[G^{q+1}]$. It is free since $(1_G,g_1,\ldots,g_q)$ with $d(1_G,g_i)\le k$ and $d(g_i,g_j)\le k$ is a subset of the free basis of $\ZZ[G^{q+1}]$ and generates $\ZZ[\Delta^q_k]$. Indeed, since $G$ is finitely generated, the number of $g_i\in G$ with $d(g_i,1_G)\le k$ is finite and so the number of basis elements of $\ZZ[\Delta^q_k]$ is finite. If $\sum a_i\sigma_i=:c\in \ZZ[\Delta^q_k]$, then $\partial c=\sum a_i\partial \sigma_i$. And if $\sigma_i=(g_0,\ldots,g_q)$, then $\partial \sigma_i=\sum_{j=0}^q(-1)^j(g_0,\ldots,\hat g_j,\ldots,g_q)$. This implies $\partial c\in \ZZ[\Delta^q_k]$. So $\partial:\ZZ[\Delta^q_k]\to \ZZ[\Delta^{q-1}_k]$ is well defined. Since $\partial^2=0$ in $\ZZ[G^{q+1}]$ already (an exact resolution is also a chain complex), the map $\partial$ is a boundary map.
\end{proof}

\begin{lem}
 A directed set of abelian groups $(A_i)_i$ is essentially trivial if and only if 
 \[
 \varinjlim_i \prod_J A_i=0
 \]
 for every index set $J$. 
\end{lem}
\begin{proof}
 If $i\le k$, suppose that the connecting map $A_i\to A_k$ is denoted by $\iota_{ik}$. In addition, $\iota_{ik}$ induces a map $\prod_JA_i\to \prod_J A_k$, which is also denoted by $\iota_{ik}$.
 
 For the direction from left to right, suppose that $(A_i)_i$ is essentially trivial. Let $J$ be an index set, and $(a_j)_{j\in J}\in \prod_J A_i$. Then there exists some $k\ge i$ such that $A_i$ vanishes in $A_k$. This means $\iota_{ik}(a_j)_{j\in J}=(\iota_{ik}(a_j))_{j\in J}=(0)_{j\in J}=0$. So $\varinjlim_i \prod_J A_i=0$.
 
 For the direction from right to left, suppose that $\varinjlim_i\prod_J A_i=0$ for every index set $J$. If $i\in I$, choose $J:=A_i$. Then $(a)_{a\in A_i}\in \prod_J A_i$. Since $\varinjlim_i \prod_J A_i=0$ there is $k\in I$ with $(\iota_{ik}(a))_{a\in A_i}=\iota_{ik}((a)_{a\in A_i})$. Then this means that $A_i$ vanishes in $A_k$. So $(A_i)_i$ is essentially trivial.
\end{proof}

If $G$ is a finitely generated group and $\chi$ a character on $G$ then the map $v:\ZZ[\Delta^q_k]\to \RR$ is defined on simplices $\sigma=(g_0,\ldots,g_q)\in \Delta^q_k$ by $v(\sigma)=\min_{0\le i\le q}\chi(g_i)$. Then $v$ maps a chain $\sum_{\sigma\in\Delta^q_k}n_\sigma \sigma$ to $\min_{\sigma:\sum n_\sigma\not=0}v(\sigma)$.

\begin{lem}
 
The map $v$ is a valuation on $(\ZZ[\Delta^q_k],\diff q)$ extending $\chi$ with $(\ZZ[\Delta^q_k],\diff q)_v=(\ZZ[\Delta^q_k\cap (G_\chi)^{q+1}],\diff q)$. The $G_\chi$-module $\ZZ[\Delta^q_k\cap (G_\chi)^{q+1}]$ is finitely generated free. If $\chi(t)<0$ then $(t^k\ZZ[\Delta^q_k\cap (G_\chi)^{q+1}],\diff q)_k$ form a filtration of $(\ZZ[G^{q+1}],\diff q)$ as chain complexes of $G_\chi$-modules.
\end{lem}
\begin{proof}
 We check \cite[Axiom~2.2]{Bieri1988} on $v$:
 \begin{align*}
  v(g\sum n(g_0,\ldots,g_q))
  &=v(\sum n(gg_0,\ldots,gg_q))\\
  &=\min_{n\not=0,i} \chi(gg_i)\\
  &=\chi(g)+\min_{n\not=0,i} \chi(g_i)\\
  &=\chi(g)+v(\sum n(g_0,\ldots,g_q)).
 \end{align*}
Now we check \cite[Axiom~2.6]{Bieri1988}:
\begin{align*}
 v(\diff q(\sum n(g_0,\ldots,g_q)))
 &=v(\sum n\sum_i (-1)^i(g_0,\ldots,\hat g_i,\ldots,g_q))\\
 &\ge \min_{n\not=0,i} \chi(g_i)\\
 &=v(\sum n(g_0,\ldots,g_q)).
\end{align*}
Here we use the fact that the vertices that are part of the boundary are part of the vertices of the chain. The other axioms are also easy to check. Then $v(c)\ge 0$ is equivalent to saying that the chain lives in $(G_\chi)^{q+1}$. Then \cite[Lemma~3.1]{Bieri1988} and Lemma~\ref{lem:fgG-module} implies that $\ZZ[\Delta^q_k\cap (G_\chi)^{q+1}]$ is finitely generated free as a $G_\chi$-module. If $(g_0,\ldots,g_q)\in G^{q+1}$, then define $k_1:=\min(\chi(g_i))/\chi(t)$, $k_2:=\max(d(g_i,g_j))$ and $k:=\lceil \max(k_1,k_2)\rceil$. Then $(g_0,\ldots,g_q)\in t^k\ZZ[\Delta^q_k\cap (G_\chi)^{q+1}]$. So, these chain complexes filter $(\ZZ[G^{q+1}],\partial_q)$.
\end{proof}

\begin{lem}
 If $F$ is a free $G$-module, then it is a flat $G_\chi$-module.
\end{lem}
\begin{proof}
 We use the characterization of flat $R$-modules which says that an $R$-module $M$ is flat if the $R$-linear relations in $M$ stem from linear relations in $R$. That is, if there is a linear relation $\sum r_ix_i=0$ with $r_i\in R,x_i\in M$, then there exist $y_j\in M,a_{i,j}\in R$ such that $\sum_i r_i a_{i,j}=0$ for every $j$ and $x_i=\sum_j a_{i,j}y_j$ for every $i$.
 
 Suppose $x_i=\sum_{gx}n_{i,gx}gx$ with $n_{i,gx}\in \ZZ,g\in G$ and $x$ elements of the basis of $F$. And $r_i=\sum_hm_{i,h}h$ with $m_{i,h}\in \ZZ,h\in G_\chi$ such that $0=\sum_ir_ix_i=\sum_i\sum_hm_{i,h}h\sum_{gx}n_{i,gx}gx$. This in particular implies $\sum_{i,h,g}m_{i,h}hn_{i,gx}g=0$ for every $x$. If $\chi=0$, then $G=G_\chi$ and we are done since free modules are flat. Otherwise, there exists some $t\in G$ with $\chi(t)<0$. Then $k\chi(t)\le \min_{n_{i,gx}\not=0}\chi(g)$ for some $k\in \NN$. Define $a_{i,x}:=\sum_gn_{i,gx}gt^{-k}$, which is in $\ZZ G_\chi$, and $y_x=t^kx\in F$. Then for every $x$ we have
 \[
  \sum_i r_i a_{i,x}=\sum_i\sum_hm_{i,h}h\sum_gn_{i,gx}gt^{-k}=\left(\sum_{i,h,g}m_{i,h}h n_{i,gx}g\right)t^{-k}=0
 \]
and 
\[
 \sum_x a_{i,x}y_x=\sum_x\sum_gn_{i,gx}gt^{-k}t^kx=\sum_{gx}n_{i,gx}gx=x_i
\]
for every $i$.

We provide an alternative proof: It suffices to show that the $G_\chi$-module $\ZZ G$ is flat. If $\chi=0$, then $G_\chi=G$ and we are done since free modules are flat. Otherwise, there exists some $t\in G$ with $\chi(t)<0$. Then $(t^n\ZZ[G_\chi])_n$ form a filtration of $\ZZ[G]$, that is, 
\[
\lim_{n\to \infty}(t^n\ZZ[G_\chi])=\ZZ[G].
\]
The $t^n\ZZ[G_\chi]$ are isomorphic to $\ZZ[G_\chi]$ which is a free $G_\chi$-module and therefore flat. Then $\ZZ[G]$, as a direct limit of flat modules, is flat itself.
\end{proof}

 \begin{thm}
 \label{thm:mainresult}
  If $\chi$ is a character on a group $G$ of type $\Flg m$, then $\chi\in\Sigma^m(G;\ZZ)$ if and only if $G_\chi$ is of type $\Flg m$.
 \end{thm}
We should again emphasize that the metric on $G_\chi$ is the induced metric from $G$. In other words, distances are measured in the Cayley graph of $G$ (with respect to a finite generating set) and not in the subgraph spanned by the vertices in $G_\chi$.
 \begin{proof}
  First, we look at the case $m=1$. In this case $\chi\in \Sigma^1(G)$ if and only if $G_\chi$ is connected as a subgraph of the Cayley graph of $G$. This happens precisely if $G_\chi$ is $1$-coarsely connected. Then $G_\chi$ is of type $\Flg 1$.  If, on the other hand, $G_\chi$ is $n$-coarsely connected for some $n\in \NN$ then if $X$ was the choice of generating set for $G$ choose $X^{\le n}$, words in $X$ of length at most $n$, as a new generating set for $G$. In the Cayley graph with this generating set, $G_\chi$ is connected. Thus, $\chi\in \Sigma^1(G)$. This is the claim for $m=1$.
 
 Now suppose $m\ge 2$ and that $G_\chi$ is of type $\Flg 1$. If $\chi=0$, then $G=G_\chi$ and we are finished. Otherwise, there exists some $t\in G$ with $\chi(t)<0$. The following is similar to \cite[Appendix~after~Chapter~3]{Bieri1988}:
\begin{align*}
 \Tor_q^{\ZZ G_\chi}(\prod \ZZ G_\chi,\ZZ)
 &= \hlgy_q((\prod \ZZ G_\chi)\otimes_{G_\chi} (\ZZ[G^{q+1}],\diff q))\\
  &=\varinjlim_k \hlgy_q((\prod\ZZ G_\chi)\otimes_{G_{\chi}}(t^k\ZZ[\Delta_k^q\cap (G_\chi)^{q+1}],\diff q))\\
  &=\varinjlim_k \hlgy_q(\prod(t^k\ZZ[\Delta_k^q\cap (G_\chi)^{q+1}],\diff q))\\
  &=\varinjlim_k \prod \hlgy_q((t^k\ZZ[\Delta_k^q\cap (G_\chi)^{q+1}],\diff q))
 \end{align*}
where we use that $(\ZZ[G^{q+1}],\diff q)$ is a free resolution of $\ZZ$ over $G$, that $\otimes$ and $\hlgy_q$ commute with $\varinjlim_k$, that $\prod$ commutes with $\otimes$ if the other factor is a finitely generated free module, and that $\prod$ commutes with $\hlgy_q$.

 Now $\Tor_q^{\ZZ G_\chi}(\prod \ZZ G_\chi,\ZZ)=0$ for $1\le j<m$ is the Bieri-Eckmann criterion for $\Flg m$ (if $\Flg 1$ is given) \cite{Bieri1974,Bieri1976}. We just need to show that $\hlgy_q((t^k\ZZ[\Delta_k^q\cap (G_\chi)^{q+1}],\diff q))_k$ is essentially trivial if and only if $\hlgy_q((\ZZ[\Delta_k^q\cap (G_\chi)^{q+1}],\diff q))_k$ is essentially trivial. 
 
 Without loss of generality, $t$ is a generator. Suppose $\hlgy_q((t^k\ZZ[\Delta_k^q\cap (G_\chi)^{q+1}],\diff q))_k$ is essentially trivial and $\hlgy_q((t^k\ZZ[\Delta_k^q\cap (G_\chi)^{q+1}],\diff q))$ vanishes in $\hlgy_q((t^l\ZZ[\Delta_l^q\cap (G_\chi)^{q+1}],\diff q))$. Let $z$ be a cycle in $\ZZ[\Delta^q_k\cap (G_\chi)^{q+1}]$. Then it also lives in $ t^k\ZZ[\Delta^q_k\cap (G_\chi)^{q+1}]$. Thus, there exists some $c\in  t^l\ZZ[\Delta^{q+1}_l\cap (G_\chi)^{q+2}]$ with $\diff {q+1} c=z$. We obtain a new sum $\tilde c$ by replacing every vertex $g_i$ in $c$ that is not part of the boundary (meaning a vertex of $\diff{q+1}c=z$) with $g_it^{-l}$. Since $\chi(g_it^{-l})\ge 0$ and
\[
 d(g_j,g_it^{-l})\le d(g_j,g_i)+d(g_i,g_it^{-l})\le 2 l
\]
and 
\[
 d(g_jt^{-l},g_it^{-l})\le d(g_jt^{-l},g_j)+ d(g_j,g_i)+d(g_i,g_it^{-l})\le 3 l,
\]
 if $(g_0,\ldots,g_{q+1})$ appears in $c$, then the sum $\tilde c$ lives in $\ZZ[\Delta^{q+1}_{3l}\cap (G_\chi)^{q+2}]$ and $\diff {q+1}\tilde c=z$. Thus, $\hlgy_q(\ZZ[\Delta^q_k\cap (G_\chi)^{q+1}])$ vanishes in $\hlgy_q(\ZZ[\Delta^q_{3l}\cap (G_\chi)^{q+1}])$. 
 
 Now suppose $\hlgy_q(\ZZ[\Delta^q_k\cap (G_\chi)^{q+1}])_k$ is essentially trivial and $\hlgy_q(\ZZ[\Delta^q_k\cap (G_\chi)^{q+1}])$ vanishes in $\hlgy_q(\ZZ[\Delta^q_l\cap (G_\chi)^{q+1}])$. Let $z$ be a cycle in $ t^k\ZZ[\Delta^q_k\cap (G_\chi)^{q+1}]$. Then $t^{-k}z$ lives in $\ZZ[\Delta^q_k\cap (G_\chi)^{q+1}]$. Thus, there exists a chain $c\in \ZZ[\Delta^{q+1}_l\cap (G_\chi)^{q+2}]$ with $\diff {q+1}c=t^{-k}z$. Then $t^kc$ lives in $t^l\ZZ[\Delta^{q+1}_l\cap(G_\chi)^{q+2}]$ with $\diff {q+1} t^kc=z$. Thus, $\hlgy_q(t^k\ZZ[\Delta^q_k\cap (G_\chi)^{q+1}])$ vanishes in $\hlgy_q(t^l\ZZ[\Delta^q_l\cap (G_\chi)^{q+1}])$.
  \end{proof}
  
  If $R$ is a commutative ring, then $\ZZ[G^{q+1}]\otimes_\ZZ R=R[G^{q+1}]$ is a free $R[G]$-resolution of the trivial module $R$. The proof of Theorem~\ref{thm:mainresult} can be adapted for this more general case. A metric space is said to be of type $\Flg m(R)$ if $\tilde \hlgy_q(\VR k X,R)_k$ is essentially trivial for $q=0,\ldots, m-1$. 
  
  \begin{cor}
   If $R$ is a commutative ring and $\chi$ is a character on a group $G$ of type $\Flg m(R)$, then $\chi\in\Sigma^m(G;R)$ if and only if $G_\chi$ is of type $\Flg m(R)$.
  \end{cor}

  Now we justify why we restrict our attention to groups $G$ of type $\Flg m$ when talking about $\Sigma^m(G;\ZZ)$.
\begin{lem}
 If $G_\chi$ is of type $\Flg m$ for some character $\chi$ then $G$ has type $\Flg m$.
\end{lem}
\begin{proof}
 First, we discuss the $m=1$ case. Suppose $G_\chi$ is of type $\Flg 1$. Then $G_\chi$ is $1$-coarsely connected for some generating set. Let $x\in G$ be an element. If $\chi=0$, then $G_\chi=G$ we are done. Otherwise, there exists some $t\in G$ with $\chi(t)>0$. Then $t^kx\in G_\chi$ for some $k$. Then there exists a path $w$ from $1$ to $t^kx$ in $G_\chi$. Then $t^{-k}w$ is a path from $1$ to $x$ in $G$. Thus, $G$ has type $\Flg 1$.
 
  Now we prove the case $m\ge 2$. Suppose $G_\chi$ has type $\Flg m$. Suppose $\homology {m-1} k {G_\chi}$ vanishes in $\homology {m-1} l {G_\chi}$ for some $l\ge k$. Let $z\in \chains{m-1} k G$ be a cycle. If $\chi=0$, then $G_\chi=G$ and we are done. Otherwise, there exists some $t\in G$ with $\chi(t)>0$. Then $t^nz\in\chains{m-1} k {G_\chi}$ for some $n\ge 0$. Then there exists some $c\in \chains m l {G_\chi}$ with $\diff m c=t^nz$. Then $t^{-n}c\in \chains m l G$ with $\diff m t^{-n}c=z$. Thus, $\homology {m-1} k G$ vanishes in $\homology {m-1} l G$.
 \end{proof}

\begin{proof}[Proof of Theorem~\ref{thma:classical}] 
If an abstract group $G$ is endowed with the discrete topology, then the compact 
sets are exactly the finite subsets of $G$. So, if $d$ is the word length metric
according to a compact generating set $X$, then $\chi\in \TopS^m(G;\ZZ)$ if and
only if $(G_\chi,d)$ is of type $\Flg m$. Now $X$ is also a finite generating
set for $G$ so $(G_\chi,d)$ is of type $\Flg m$ if and only if $\chi\in 
\Sigma^m(G;\ZZ)$ by Theorem~\ref{thm:mainresult}.
\end{proof}

%% file: criteria.tex
\section{Criteria for homological compactness properties}
\label{sec:criteria}
In this section we show that a witness for homological compactness properties is a chain endomorphism. Namely, we prove Theorem~\ref{thma:mu}, Theorem~\ref{thma:varphi} and Theorem~\ref{thma:crit2}.

If $G$ is a compactly generated, locally compact, Hausdorff group and $\chi$ a character on $G$ then the map $v:\ZZ[\Delta^q_k]\to \RR$ is defined on simplices $\sigma=(g_0,\ldots,g_q)\in \Delta^q_k$ by $v(\sigma)=\min_{0\le i\le q}\chi(g_i)$. For a chain $c:=\sum_{\sigma\in\Delta^q_k}n_\sigma \sigma$, we define $v(c)=\min_{n_\sigma\not=0}v(\sigma)$.

 If $X$ is an infinite set (we consider $X$ to be a placeholder for values in $G$) then any element $S\in \ZZ[X^{q+1}]$ (here the notation $
\ZZ[X^{q+1}]$ stands for the free abelian group on $X\times\cdots\times X$) is called a \emph{$q$-shape}. We define
\[
 S^{(0)}:=\{x\in X\mid \exists \sigma\in \supp S\mbox{ with $x$ a vertex in $\sigma$}\}
\]
and
\[
 S^{(1)}:=\{(x,y)\in X\times X\mid \exists \sigma\in \supp S\mbox{ with $x,y$ vertices in $\sigma$}\}.
\]
We say that a shape $S$ is \emph{connected} if $(S^{(0)},S^{(1)})$ forms a connected graph. We say $S$ is \emph{centric} if $e\in S^{(1)}\cap ((\partial S)^{(0)}\times (\partial S)^{(0)})$ implies $e\in (\partial S)^{(1)}$. We say $S$ is \emph{nondegenerate} if every simplex in $\supp(S)$ is nondegenerate.
\begin{lem}
\label{lem:connectedshape}
\begin{enumerate}
 \item Suppose $S,T$ are shapes with $\partial_q S=T$, $T$ is connected and $S=S_1+\cdots+S_n$ is the sum of its connected components. Then $\partial_q S_i=T$ for some $i$.
 \item If $S$ is a $1$-shape with $\diff 1 S=x_1-x_0$, then there exists a connected component $S_i$ of $S$ with $\diff 1 S_i=x_1-x_0$.
 \item If $S$ is a nondegenerate connected shape, then there exists a nondegenerate connected centric shape $S'$ with $\partial S=\partial S'$.
 \end{enumerate}
\end{lem}
\begin{proof}
 First, we prove claim 1. In fact, we will show that $\partial S_i=0$ for all except at most one $i\in\{1,\ldots,n\}$. Assume for contradiction that $i,j\in\{1,\ldots,n\}$ are distinct with $\partial_q S_i\not=0$ and $\partial_q S_j\not = 0$. So, there exist $x_i\in \supp\partial_q S_i$ and $x_j\in\supp \partial_q S_j$. Thus, $x_i,x_j\in T^{(0)}$. Now $x_i,x_j$ are not connected in $S$, so they are not connected in $\partial_q S=T$. This is in contradiction to $T$ being connected.
 
 Now we prove claim 2. Suppose $S=S_1+\cdots+S_n$ is the sum of its connected components and $x_0\in (\diff 1 S_i)^{(0)}$. Since $\im\diff 1=\ker \diff 0$, where $\diff 0:\ZZ[X]\to \ZZ$ is the augmentation map, there exists $y\in (\diff 1 S_i)^{(0)}$ with $y\not=x_0$. Now, the $(\diff 1 S_j)^{(0)}$ are all pairwise disjoint, so $y=x_1$ and $\diff 1 S_j=0$ for $j\not=i$. This implies $\diff 1 S_i=x_1-x_0$.
 
 We now prove claim 3 by induction on the number of pairs of bad vertices $x_0,x_1$ in $(\partial S)^{(0)}$ that are connected by an edge in $S^{(1)}$ but not in $(\partial S)^{(1)}$. Suppose $x_0,x_1\in (\partial S)^{(0)}$ and $(x_0, x_1)\not\in (\partial S)^{(1)}$. If we view $S,\partial S$ as functions $\Delta^q\to \ZZ$ and $\Delta^{q-1}\to \ZZ$ then this implies
 \[
  \sum_{x_0,x_1\in \sigma}(\partial S)(\sigma)\sigma=0.
 \]
 Choose distinct vertices $y_0,y_1\in X$ not in $S^{(0)}$. Then the shape $S_1'$ is obtained from $S$ by replacing $x_0$ with $y_0$, the shape $S_0'$ is obtained from $S$ by replacing $x_1$ with $y_1$ and the shape $S'_2$ is obtained from $S$ by replacing $x_0$ with $y_0$ and $x_1$ with $y_1$. Then we define   
 \[
 S':=S_0'+S_1'-S_2'.
 \]
 It remains to show that $\partial S=\partial S'$. For that, we split $\partial S'$ into sums.
 \begin{align*}
  \partial S'&=
  \partial (S_0'+S_1'-S_2')\\
  &=\left(\sum_{x_0\in \sigma,y_1\not\in\sigma}(\partial S_0')(\sigma)\sigma+\sum_{x_0\not\in \sigma,y_1\in\sigma}(\partial S_0')(\sigma)\sigma
  +\sum_{x_0,y_1\not\in \sigma}(\partial S_0')(\sigma)\right)\\ 
  &+\left(\sum_{x_1\in \sigma,y_0\not\in\sigma}(\partial S_1')(\sigma)\sigma+\sum_{x_1\not\in \sigma,y_0\in\sigma}(\partial S_1')(\sigma)\sigma
  +\sum_{x_1,y_0\not\in \sigma}(\partial S_1')(\sigma)\right)\\
  &-\left(\sum_{y_0\in \sigma,y_1\not\in\sigma}(\partial S_2')(\sigma)\sigma+\sum_{y_0\not\in \sigma,y_1\in\sigma}(\partial S_2')(\sigma)\sigma
  +\sum_{y_0,y_1\not\in \sigma}(\partial S_2')(\sigma)\right)\\
  &=\sum_{x_0\in \sigma,x_1\not\in\sigma}(\partial S)(\sigma)\sigma+\sum_{x_0\not\in \sigma,x_1\in\sigma}(\partial S)(\sigma)\sigma+\sum_{x_0, x_1\not\in\sigma}(\partial S)(\sigma)\sigma\\
  &=\partial S.
 \end{align*}
In $S'$ the vertices $x_0$ and $x_1$ are not joined by an edge. And $y_0,y_1$ do not appear in $\partial S'$. So we reduced the number of bad vertices by one.
\end{proof}

If $S=\sum n(x_0,\ldots,x_q)$ is a shape then $c\in \ZZ[G^{q+1}]$ is said to \emph{have shape $S$} if there exists a mapping $\varphi:S^{(0)}\to G$ with
\[
 c=\sum n(\varphi(x_0),\ldots,\varphi(x_q)).
\]
We say $c=(S,\varphi)$ has shape $S$ with vertices $\varphi$. Note that $(S,\varphi)$ uniquely determines $c$. On the other hand, there are many shapes that $c$ can have and $\varphi$ is also not uniquely determined by $S,c$. Conversely, given a chain $c\in \ZZ[G^{q+1}]$ it has some shape, say $S$, with injective vertex map $c^{(0)}:S^{(0)}\to G$. With this extra condition, we say \emph{$S$ is the shape of $c$}. This shape is unique up to relabeling. Note that the shape $S$ being connected is intrinsically a property of the chain $c$. Also note that a chain homomorphism does not necessarily map connected shapes to connected shapes (contrary to the intuition from continuous maps).
 
 If $k\ge 0,$ $m\in \NN$ and $\varepsilon_*:\chain_*(\VR k G)^{(m)}\to \ZZ[G^{*+1}]$ is a $\ZZ G$-chain map extending the identity on $\ZZ$, then $\varepsilon_*$ is said to be \emph{finitely modeled} if 
 \begin{itemize}
  \item if $q=0$ then $\varepsilon_0(1_G)=t$ for some $t\in G$;
  \item if $q=1$ then there are connected, non-degenerate $1$-shapes $S^1_1,\ldots,S^1_n$ such that for every $\sigma\in\Delta^1_k\cap(\{1_G\}\times G)$ there exists some $i\in\{1,\ldots,n\}$ such that $\varepsilon_1(\sigma)$ has shape $S_i^1$;
  \item if $2\le q\le m$ then there are connected, nondegenerate, centric shapes $S^q_1,\ldots,S^q_n$ such that for every $\sigma\in \Delta^q_k\cap (\{1_G\}\times G^q)$ we can write
  \[
   \varepsilon_{q-1}\circ \partial_q(\sigma)=c_1+\cdots+c_l
  \]
as sum of connected components, and 
\[
 \varepsilon_q(\sigma)=d_1+\cdots+d_l
\]
such that for each $j=1,\ldots,l$ we have 
\begin{itemize}
\item $\partial_q d_j=c_j$;
\item there is some $i\in \{1,\ldots,n\}$ such that $d_j$ has shape $S_i^q$.
\end{itemize}
\item  $\varepsilon(\tau)^{(0)}\setminus (\varepsilon\circ\partial(\tau))^{(0)}$, over all at most $m$-dimensional simplices $\tau$, is $G$-finite.
 \end{itemize}
\begin{rem}
    We now explain the intuition of the property finitely modeled. The definition is motivated by the discrete case in which a chain endomorphism $\varepsilon:\chain_q(\VR k G)^{(m)}\to \ZZ[G^{q+1}]$ can be specified on a finite $\ZZ G$-basis. The homotopy that joins $\varepsilon$ with the identity in $\ZZ[G^{*+1}]$ therefore has image in $\chain_q(\VR l G)^{(m)}$ for some $l\ge k$. In the locally compact case, we have to impose this property by construction (a version of $G$-finiteness of the image). To this end, we introduced the fairly geometric notion of a shape on chains. In the homotopical setting this is associated with a (barycentric) subdivision of a simplex. Also in the homotopical setting the image of a connected simplicial set under a simplicial map is connected again. This is not the case for shapes and chain endomorphisms, so we have to impose some version of connectedness as well.
\end{rem}
 \begin{lem}
 \label{lem:homotopy_homology_locallycompact}
  If $k\le l$ and $\varepsilon_*:{\chains * k G}^{(m)}\to {\chains * l G}^{(m)}$ is a finitely modeled $\ZZ G$-chain endomorphism extending the identity on $\ZZ$ then 
  \begin{itemize}
   \item there exists $K\ge 0$ such that for every $0\le q \le m$, $\sigma\in (\{1_G\}\times G^q)\cap \Delta^q_k$ and $h\in \varepsilon_q(\sigma)^{(0)}$ we have $\ell(h)\le K$;
   \item there exists a chain homotopy $\eta_*:{\chains * k G}^{(m)}\to \ZZ[G^{*+2}]^{(m+1)}$ joining $\varepsilon_*$ to the inclusion $\id_*:{\chains * k G}^{(m)}\to {\chains * l G}^{(m)}$ such that
    \begin{enumerate}
    \item for every $\tau\in \Delta^m_k$ we have
    \[
    \eta_m(\tau)^{(0)}\subseteq \bigcup_{q=0}^m\bigcup_{\substack{\sigma\le \tau,\\ [\tau:\sigma]=m-q}} (\varepsilon_q(\sigma))^{(0)}\cup \tau^{(0)}.
   \]
   \item $\im \eta_*\subseteq {\chains {*+1} L G}^{(m+1)}$ for some $L\ge 0$.
  \end{enumerate}
  \end{itemize}
 \end{lem}
\begin{proof}
We first show the first claim. We proceed by induction on $q=0,\ldots, m$ starting with $q=0$. Since $\varepsilon$ is finitely modeled, $h\in \varepsilon_0(1_G)^{(0)}$ implies $h=t$ so $\ell(h)\le \ell(t)$. If $q=1$ then since $\varepsilon$ is finitely modeled there are only finitely many connected shapes $S$, each with $\diam(S^{(0)},S^{(1)})$ less than say $k_1$, which $\varepsilon_1(\sigma)$ can have. If $g_1=1_G$, then $\varepsilon_0\circ\partial_1(\sigma)=0$. So $\varepsilon_1(\sigma)=0$ as well. If $g_1\not=1_G$ then $t\in \varepsilon_0\circ\partial_1(1_G,g_1)$. So $\ell(h)\le \ell(t)+lk_1$. If $q\ge 2$ we show inductively that $\ell(h)\le (q-1)k+\ell(t)+\ell(k_1+k_2+\cdots+k_q)$, where $k_q$ is the maximum of $\diam(S^{(0)},S^{(1)})$ over all shapes $S$ that $d_i$ can have (over all $\sigma\in \Delta^q_k$) where
\[
 \varepsilon_q(\sigma)=d_1+\ldots+d_n
\]
and 
\[
 \varepsilon_{q-1}\circ\partial_q(\sigma)=c_1+\cdots+c_n
\]
sum of its connected components with $\partial_qd_i=c_i$ for all $i=1,\ldots,n$. Then there exists some $i\in \{1,\ldots,n\}$ with $h\in d_i^{(0)}\supseteq c_i^{(0)}\not=\emptyset$. Suppose $g\in c_i^{(0)}$. Then $d(g,h)\le lk_q$ and 
\[
g\in \varepsilon_{q-1}\circ\partial_q(\sigma)\subseteq\varepsilon_{q-1}(\sigma_0)^{(0)}\cup \varepsilon_{q-1}(\sigma_1)^{(0)}\cup \cdots \cup \varepsilon_{q-1}(\sigma_q)^{(0)}
\]
If $g\in \varepsilon_{q-1}(\sigma_0)^{(0)}$ then $d(g,f)\le k$ where $f\in 
\varepsilon_{q-1}(g_1^{-1}\sigma_0)$ and if $g\in 
\varepsilon_{q-1}(\sigma_1)^{(0)}\cup\cdots\cup 
\varepsilon_{q-1}(\sigma_q)^{(0)}$ we can also use the induction hypothesis (set 
$f=g$) which yields
\begin{align*}
 \ell(h)
 &\le d(g,h)+d(g,f)+\ell(f)\\
 &\le lk_q+k+\ell(f)\\
 &\le lk_q+k+(q-2)k+\ell(t)+\ell(k_1+k_2+\cdots+k_{q-1})\\
 &=(q-1)k+\ell(t)+\ell(k_1+k_2+\cdots+k_q)
\end{align*}
If we set $K:=k(m-1)+\ell(t)+\ell(k_1+\cdots+k_m)$, we obtain the result.

 Now, we present a proof for the second claim. We construct the chain homotopy 
$\eta_q$ inductively starting with $q=-2$. For each $q\ge -1$ set
$\lambda_q:=\varepsilon_q-\id_q$. Since $\lambda_{-1}=0$, we are forced to set 
$\eta_{-2}:=0,\eta_{-1}:=0$. Then
 \[
  \lambda_{-1}=0=\eta_{-2}\circ \partial_{-1}+\partial_0\circ \eta_{-1}.
 \]
Now suppose $m\ge q\ge 0$ and $\eta_{-2},\ldots,\eta_{q-1}$ have been constructed. We prove $\mbox{im}(\lambda_q-\eta_{q-1}\circ\diff q)\subseteq \ker \partial_q$:
\begin{align*}
 \partial_q\circ(\lambda_q-\eta_{q-1}\circ \partial_q)
 &=\partial_q\circ\lambda_q-\partial_q\circ \eta_{q-1}\circ \partial_q\\
 &=\partial_q\circ\lambda_q-(\lambda_{q-1}-\eta_{q-2}\circ \partial_{q-1})\circ\partial_q\\
 &=\partial_q\circ\lambda_q-\lambda_{q-1}\circ\partial_q\\
 &=0
\end{align*}
If $\sigma:=(1_G,g_1,\ldots,g_q)\in \Delta^q_k$, then $\sum n (x_0,\ldots,x_q):=(\lambda_q-\eta_{q-1}\circ \partial_q)(\sigma)\in \ker\partial_q$. Then $\eta_q(\sigma):=\sum n(1_G,x_0,\ldots,x_q)$ has the property
\begin{align*}
 \partial_{q+1}\circ \eta_q(\sigma)
 &=\partial_{q+1}\sum n(1_G,x_0,\ldots,x_q)\\
 &=\sum n(x_0,\ldots,x_q)-\sum n\sum_{i=0}^q(-1)^i(1_G,x_0,\ldots,\hat 
x_i,\ldots,x_q)\\
&=\sum n(x_0,\ldots,x_q)-1_G\times \partial_q\left(\sum 
n(x_0,\ldots,x_q)\right)\\
 &=\sum n(x_0,\ldots,x_q)\\
&=\lambda_q-\eta_{q-1}\circ\partial_q(\sigma).
\end{align*}
If $\tau:=(g_0,\ldots,g_q)\in \Delta^q_k$, then $\sigma:=(1_G,g_0^{-1}g_1,\ldots,g_0^{-1}g_q)$ is a simplex with $\tau=g_0\sigma$. We set $\eta_q(\tau):=g_0\eta(\sigma)$. Then
\begin{align*}
 (\varepsilon_q-\id_q)(\tau)
 &=g_0(\varepsilon_q-\id_q)(\sigma)\\
 &=g_0\lambda_q(\sigma)\\
 &=g_0(\partial_{q+1}\circ\eta_q(\sigma)+\eta_{q-1}\circ\partial_q(\sigma))\\
 &=\partial_{q+1}\circ\eta_q(\tau)+\eta_{q-1}\circ\partial_q(\tau).
\end{align*}
So $\eta_q$ joins $\varepsilon_q$ to $\id_q$. It remains to show the bound on the image of $\eta_q$. We show 1. of the last statement. If $(g_0,\ldots,g_m)=\tau\in\Delta^m_q$, then $\tau':=g_0^{-1}\tau$ is an element of the basis. Since $\eta_{-1}=0$, the claim is obviously true for $m=-1$. If $m>-1$ and we know that the claim is true for $m-1$, then
\begin{align*}
 \eta_m(\tau)^{(0)}
 &=g_0(\eta_m(\tau')^{(0)})\\
 &\subseteq g_0(\{1_G\}\cup \tau'^{(0)}\cup \varepsilon_m(\tau')^{(0)}\cup \eta_{m-1}\circ\partial_m(\tau')^{(0)})\\
 &=\{g_0\}\cup \tau^{(0)}\cup \varepsilon_m(\tau)^{(0)}\cup \eta_{m-1}\circ\partial_m(\tau)^{(0)}\\
 &\subseteq \tau^{(0)}\cup \varepsilon_m(\tau)^{(0)}\cup 
\bigcup_{\substack{\sigma\le \tau,\\ 
[\tau:\sigma]=1}}\varepsilon_{m-1}(\sigma)\\
 &\subseteq \tau^{(0)}\cup \bigcup_{q=0}^m\bigcup_{\substack{\sigma\le 
\tau,\\ [\tau:\sigma]=q
}} (\varepsilon_q(\sigma))^{(0)}.
\end{align*}
Now we show $2$ of the last claim. By the first claim, there exists 
some $K\ge 0$ such that $\ell(x)\le K$ for every $x\in \varepsilon_q(\sigma)^{(0)}, 
\sigma\in (\{1_G\}\times G^q)\cap \Delta^q_k$ and $q=0,\ldots,m$. So, if 
$\tau=(1_G,g_1,\ldots,g_m)$ and $y\in \eta_m(\tau)^{(0)}$, then $\ell(y)\le 
\max(k,k+K)=k+K$. This implies $\eta_m(\tau)\in \chain_{m+1}(\VR{2(K+k)}G)$ and
ultimately $\im\eta_m\subseteq \chains{m+1} {2(K+k)} G$.
\end{proof}

 The set of chains with shape $S$ is equipped with a topology in the following way. If $\varphi$ is an assignment of vertices, then its chain represents a point $(\varphi(s))_{s\in S^{(0)}}\in \prod_{S^{(0)}}G$. And $\prod_{s\in S^{(0)}}G$ is equipped with the product topology of copies of $G$ that are equipped with the locally compact topology assigned to them.

If $G$ is a locally compact group, then $(k_0,\ldots,k_m)\in\NN^{m+1}$ is said to
be a homological connecting vector for $G$ if $\tilde \hlgy_q(\VR{k_q+1}G)$ 
vanishes in $\tilde \hlgy_q(\VR{k_{q+1}}G)$ for every $q=0,\ldots,m-1$.

\begin{thm}
\label{thm:mu}
If $m\ge 1$ and $G$ is a locally compact group then the following are equivalent:
\begin{enumerate}
 \item $G$ is of type $\Clg m$;
 \item there exists a homological connecting vector $(k_0,\ldots,k_m)$ for $G$;
 \item there exists $K_0\ge 0$ such that for every $k\ge K_0$ there exists a finitely modeled $\ZZ G$-chain endomorphism 
 \[
 \mu_*:{\chains * k G}^{(m)}\to {\chains * {K_0} G}^{(m)}
 \]
 extending the identity on $\ZZ$.
\end{enumerate}
\end{thm}
\begin{rem}
\label{rem:cg}
 Note that the existence of a metric on $G$ inducing the coarse structure $(\Delta_C)_{C\in \mathcal C(G)}$ is implicit in Theorem~\ref{thm:mu}. For if $(\Delta_C)_{C\in \mathcal C(G)}$ does not induce a metric, then in particular there is no compact generating set for $G$. This implies the negation of all 3 conditions. Then $G$ is not of type $\Clg 1$ so in particular not of type $\Clg m$ for $m\ge 1$. That is, the negation of condition 1. That $G$ is not compactly generated also implies that $\tilde \hlgy_0(\ZZ[\Delta^1_C])$ does not vanish for every $C\in \mathcal C(G)$. That is, the negation of condition 2. Assume for contradiction condition 3. There exists some compact $\tilde K_0$ and a chain endomorphism $\mu$ with $\mu_0(1_G)=:t$ and $\im \mu_1\subseteq \ZZ[\Delta_{\tilde K_0}]$. Then for every $x,y\in G$ there exists some $k\ge 0$ with $(x_t^{-1},y_t^{-1})\in \Delta^1_k$. Then
 \[
  \partial_1 \mu_1(xt^{-1},yt^{-1})=\mu_0\circ \partial_1(xt^{-1},yt^{-1})=y-x.
 \]
Thus, the data of $\mu_1(xt^{-1},yt^{-1})$ describes a $\tilde K_0$-path that joins $x$ to $y$. In this way, $G$ is compactly generated by $\tilde K_0$. That is a contradiction, which implies the negation of condition 3.
\end{rem}

 \begin{proof}
As discussed in Remark~\ref{rem:cg} we can assume $G$ is compactly generated and therefore metrizable.
 
 Suppose condition 1, that $G$ is of type $\Clg m$. Then it is easy to see that
there exists a homological connecting vector $(k_0,\ldots,k_m)$ for $G$. So 
condition 1 implies condition 2. 
 
 Now we show that condition 2 implies condition 3. Suppose condition 2. That is, 
there exists a homological connecting vector $(k_0,\ldots,k_m)$ for $G$. We
show condition 3: Suppose $\mathcal X$ is both a compact neighborhood of $1_G$ 
and a compact generating set for $G$ and the metric on $G$ is the word length 
metric induced by $\mathcal X$. Choose an open set $U$ with $1_G\in U\subseteq \mathcal X$. We construct $\mu_q$ inductively with $\im \mu_q\subseteq \chains q {k_q+1} G$ for every $q=0,\ldots,m$. If $q=0$, then 
$\mu_0:\ZZ[G]\to \ZZ[G]$ is defined to be the identity. If 
$\mu_0,\ldots,\mu_{q-1}$ have been constructed, then by the assumption on 
$\mu_{q-1}$ there exist shapes $T_1^q,\ldots,T_n^q$ such that for every $\tau\in\Delta^q_k$ there exists some $i\in \{1,\ldots,n\}$ so that $\mu_{q-1}\circ\partial_q(\tau)$ has shape $T_i^q$. Fix one such shape $T:=T_i^q$.

Suppose that we are given a pair $(S,c)$ 
\begin{itemize}
\item where $S$ is a shape with $\partial_q S=T$;
\item denote the inner vertices of $S$ by $R^{(0)}:=S^{(0)}\setminus T^{(0)}$, then $c$ is a mapping $R^{(0)}\to G$ with 
\[
d(c(r_1),c(r_2))\le k_q
\]
for every $(r_1,r_2)\in (R^{(0)}\times R^{(0)})\cap S^{(1)}$.
\item if $q=1$ then $S$ is non-degenerate, centric and connected;
\item if $q\ge 2$ and $T=T_1+\cdots+T_n$ the sum of its connected components then there exist nondegenerate, centric and connected shapes $S_1,\ldots,S_n$ such that $S=S_1+\cdots+S_n$ and $\partial_q S_i=T_i$ for every $i=1,\ldots,n$.
\end{itemize}

We denote by $C_{S,c}$ the set of chains $(T,\varphi)$ of shape $T$ 
with vertices $\varphi:T^{(0)}\to G$ such that for every $(r,t)\in S^{(1)}\cap(R^{(0)}\times T^{(0)})$:
\[
 d(c(r),\varphi(t))\le k_q.
\]
Then 
\[
C_{S,c}=\prod_{t\in T^{(0)}}\bigcap_{\substack{r\in R^{(0)},\\ (r,t)\in S^{(1)}}}c(r)\mathcal X^{k_q}.
 \]
 Then we define the open set
 \begin{align*}
 U_{S,c}
 &:=\{(g_th_t)_t\in \prod_{t\in T^{(0)}}G\mid (g_t)_t\in C_{S,c}\mbox{ and }(h_t)_t\in \prod_{t\in T^{(0)}}U\}\\
 &\subseteq \prod_{t\in T^{(0)}}
 \bigcap_{\substack{r\in R^{(0)},\\ (r,t)\in S^{(1)}}}c(r)\mathcal X^{k_q+1}.
\end{align*}

Define
\[
 \mathcal M:=\{c\in\ZZ[G^q]\mid\exists \sigma\in (\{1_G\}\times G^q)\cap\Delta^q_k: c=\mu_{q-1}\circ\partial_q(\sigma)\mbox{ and $c$ has shape }T\}.
\]
We show that $\mathcal M$ has compact closure in the set of chains of shape $T$ topologized as $\prod_{t\in T^{(0)}}G$. Since $\mu_0,\ldots,\mu_{q-1}$ defines a finitely modeled chain endomorphism on ${\chains * k G}^{(q-1)}$, we can apply Lemma~\ref{lem:homotopy_homology_locallycompact}. 
Thus, there exists $K\ge 0$ such that for every $\tau\in (\{1_G\}\times G^{q-1})\cap \Delta^{q-1}_k$ and $x\in 
\mu_{q-1}(\tau)^{(0)}$ we have $\ell(x)\le K$. 
Then for every $\sigma\in (\{1_G\}\times 
G^q)\cap\Delta^q_k$ and $y\in \mu_{q-1}\circ\partial_q(\sigma)^{(0)}$ we have $\ell(y)\le K+k$. So $\mathcal M\subseteq \prod \mathcal X^{K+k}$ as a subset of a compact set has compact closure $\bar{\mathcal M}$. 

If $\varphi\in \partial\mathcal 
M$, a point on the boundary of $\mathcal M$, then there exists some $\varphi'\in\mathcal M$ with $\varphi\in \prod_{t\in T^{(0)}}\varphi'(t)U$. If $\varphi'\in C_{S,c}$, then $\varphi\in U_{S,c}$. Since by definition of the connecting vector we have $\tilde \hlgy_{q-1}(\VR{k_{q-1}+1}G)$ vanishing in $\tilde \hlgy_{q-1}(\VR{k_q}G)$, each $\varphi'\in \mathcal M$ is contained in some $C_{S,c}$ and by Lemma~\ref{lem:connectedshape} we can assume $S$ to be of the form described in the above list. Thus, $U_{S,c}$ over all $S,c$ are an open cover of $\bar{\mathcal M}$. 

Since $\bar {\mathcal M}$ was shown to be compact, there exists $S_1,c_1,\ldots, S_n,c_n$ with
\[
 U_{S_1,c_1}\cup\cdots \cup U_{S_n,c_n}\supseteq\bar {\mathcal M}.
\]
So, if $\sigma\in (\{1_G\}\times G^q)\cap \Delta^q_k$, then there exists $i\in\{1,\ldots,n\}$ with $\mu_{q-1}\circ\partial_q(\sigma)\in U_{S_i,c_i}$. Define $\mu_q(\sigma)$ to be the chain of shape $S_i$ with vertices
\begin{align*}
 \mu_q(\sigma)^{(0)}:S_i^{(0)}&\to G\\
 s&\mapsto \begin{cases}
            \mu_{q-1}\circ\partial_q(\sigma)^{(0)}(s) & s\in T^{(0)}\\
            c_i(s) & s\in R_i^{(0)}.
           \end{cases}
\end{align*}
Since $(\{1_G\}\times G^q)\cap \Delta^q_k$ is a $\ZZ G$-basis for $\chains q k G$, we have defined $\mu_q:{\chains q k G}^{(q)}\to \chains q {k_q+1} G$ as a $\ZZ G$-chain endomorphism that is finitely modeled. If $m=q$, then $K_0:=k_m+1$ is the required constant.
  
Now we show that condition 3 implies condition 1. Suppose that we have a chain endomorphism $\mu_*$ with the required properties. We now show that $G$ is of type $\Clg m$. Suppose that we already showed that $G$ is of type $\Clg {m-1}$. Since $\mu_*$ extends the identity on $\ZZ$ and is finitely modeled there exist by Lemma~\ref{lem:homotopy_homology_locallycompact} a number $l\ge 0$ and a chain homotopy $\eta_*:{\chains * k G}^{(m)} \to {\chains {*+1} l G}^{(m+1)}$ joining $\id$ to $\mu$. If $z\in \chains {m-1} k G$ is a cycle then since $(\ZZ[G^{q+1}],\partial_q)$ is acyclic there exists a chain $c\in \ZZ[G^{m+1}]$ with $\partial_m c=z$. Then the chain $\mu_m(c)+\eta_{m-1}(z)$ lies in $\chain_m(\VR {\max(l,K_0)} G)$ and
\begin{align*}
 \partial_m(\mu_m(c)+\eta_{m-1}(z))
 &=\mu_{m-1}\circ \partial_m(c)+\partial_m\circ\eta_{m-1}(z)\\
 &=\mu_{m-1}(z)+(z-\mu_{m-1}(z))\\
 &=z.
\end{align*}
 This shows that $z$ is a boundary in $\chain_*(\VR {\max(l,K_0)} G)$. Thus, $\homology {m-1} k G$ vanishes in $\homology {m-1} {\max(l,K_0)} G$. Thus, $G$ is of type $\Clg m$.
 \end{proof}
 
\begin{thm}
\label{thm:sigmacrit_lc_homol}
 Let $m\in \NN$. If $G$ is a group of type $\Clg m$ and $\chi:G\to \RR$ a non-zero character then the following are equivalent:
 \begin{enumerate}
  \item $\chi\in \TopS^m(G,\ZZ)$;
  \item there is $(k_0,\ldots,k_m)\in \NN^{m+1}$ such that $\tilde 
\hlgy_q(\VR{k_q+1}{G_\chi})$ vanishes in $\tilde \hlgy_q(\VR{k_{q+1}}{G_\chi})$ 
for every $q=0,\ldots,m-1$;
  \item for every $k\ge 0$ large enough there is a finitely modeled chain endomorphism 
  \[
  \varphi_*:{\chains * k G}^{(m)}\to {\chains * k G}^{(m)}
  \]
  of $\ZZ G$-complexes extending the identity on $\ZZ$ and $K> 0$ such that for every $q=0,\ldots,m$ and $c\in \chains q k G$ we have
  \[
  v(\varphi_q(c))-v(c)\ge K.
  \]
 \end{enumerate}
\end{thm}
\begin{proof}
 It is obvious that condition 1 implies condition 2.

We now prove that condition 2 implies condition 3. Suppose condition 2, that there exists $(k_0,\ldots,k_m)\in\NN^{m+1}$ such that $\tilde \hlgy_q(\VR{k_q+1}{G_\chi})$ vanishes in $\tilde \hlgy_q(\VR{k_{q+1}}{G_\chi})$. Suppose $\mathcal X$ is both a compact neighborhood of $1_G$ and a compact generating set for $G$ and $G$ is equipped with the word length metric induced by $\mathcal X$. Choose an open set $U$ with $1_G\in U\subseteq \mathcal X$. Since $U$ is relatively compact and $\chi$ is continuous
\[
 K:=\sup\{|\chi(u)|\mid u\in U\}
\]
is finite. Then there exists $\tilde t\in G$ with $\chi(\tilde t)\ge (m+1)K$. 

We are going to construct $\varphi_q$ inductively for every $q=0,\ldots,m$ with $v(\varphi_q(\tau))-v(\tau)\ge (m+1-q)K$ for every $\tau\in\Delta^q_k$. If $q=0$, then $\varphi_0:\ZZ[G]\to \ZZ[G]$ is defined to map $1_G\mapsto t$. Then
\[
 v(\varphi_0(g))-v(g)=\chi(gt)-\chi(g)=\chi(t)\ge (m+1)K.
\]

If $\varphi_0,\ldots,\varphi_{q-1}$ have been constructed, then by the assumption on $\varphi_{q-1}$ there exist shapes $T_1^q,\ldots,T_n^q$ such that for every $\tau\in\Delta^q_k$ there exists some $i\in \{1,\ldots,n\}$ so that $\varphi_{q-1}\circ\partial_q(\tau)$ has shape $T_i^q$. Fix one such shape $T:=T_i^q$.

Suppose that we are given a pair $(S,c)$ where 
\begin{itemize}
\item $S$ is a shape with $\partial_q S=T$;
\item denote the inner vertices of $S$ by $R^{(0)}:=S^{(0)}\setminus T^{(0)}$, then $c$ is a mapping $R^{(0)}\to G$ with 
\[
d(c(r_1),c(r_2))\le k_q
\]
for every $(r_1,r_2)\in (R^{(0)}\times R^{(0)})\cap S^{(1)}$.
\item if $q=1$ then $S$ is non-degenerate, centric and connected;
\item if $q\ge 2$ and $T=T_1+\cdots+T_n$ the sum of its connected components then there exist nondegenerate, centric and connected shapes $S_1,\ldots,S_n$ such that $S=S_1+\cdots+S_n$ and $\partial_q S_i=T_i$ for every $i=1,\ldots,n$.
\end{itemize}
We denote by $C_{S,c}$ the set of chains $(T,\mu)$ having shape $T$ with vertices $\mu:T^{(0)}\to G$ such that
 \begin{enumerate}
 \item for every $(r,t)\in S^{(1)}\cap(R^{(0)}\times T^{(0)})$:
\[
 d(c(r),\mu(t))\le k_q;
\]
\item for every $r\in R^{(0)}$ there exists $t\in T^{(0)}$ with
\[
 \chi(c(r))\ge \chi(\mu(t)).
\]
\end{enumerate}
Then $C_{S,c}$ is the intersection of two sets $C_1,C_2$ defined by
 \[
  C_1:=\prod_{t\in T^{(0)}}\bigcap_{\substack{r\in R^{(0)},\\ (r,t)\in S^{(1)}}}c(r)\mathcal X^{k_q}
 \]
and 
\[
 C_2:=\bigcap_{r\in R^{(0)}}\bigcup_{t\in T^{(0)}}\left(\prod_{T^{(0)}\ni s\not=t}G\right)\times \chi^{-1}(-\infty,\chi(c(r))]
\]
 Then we define the open set
 \begin{align*}
 U_{S,c}
&:=\{(g_th_t)_t\in \prod_{t\in T^{(0)}}G\mid (g_t)_t\in C_{S,c};(h_t)_t\in \prod_{t\in T^{(0)}}U\}\\
 &\subseteq \bigcap_{i=1,2}\{(g_th_t)_t\in \prod_{t\in T^{(0)}}G\mid (g_t)_t\in C_i;(h_t)_t\in \prod_{t\in T^{(0)}}U\}\\
 &\subseteq \left(\prod_{t\in T^{(0)}}\bigcap_{\substack{r\in R^{(0)},\\ (r,t)\in S^{(1)}}}c(r)\mathcal X^{k_q+1}\right)\\
 &\cap\left(\bigcap_{r\in R^{(0)}}\bigcup_{t\in T^{(0)}}\left(\prod_{T^{(0)}\ni s\not=t}G\right)\times \chi^{-1}(-\infty,\chi(c(r))-K]\right).
\end{align*}
Define
\[
 \mathcal M:=\{c\in\ZZ[G^q]\mid c=\varphi_{q-1}\circ\partial_q(\sigma)\mbox{ of shape }T\mbox{ for some }\sigma\in (\{1_G\}\times G^q)\cap\Delta^q_k\}.
\]
We show that $\mathcal M$ has compact closure in the set of chains of shape $T$ 
topologized by $\prod_{t\in T^{(0)}}G$. Since $\varphi_0,\ldots,\varphi_{q-1}$ define a finitely modeled chain endomorphism on ${\chains * k G}^{(q-1)}$, we can apply 
Lemma~\ref{lem:homotopy_homology_locallycompact}. Thus, there exists $L\ge 0$ such that for every  $\tau\in 
(\{1_G\}\times G^{q-1})\cap \Delta^{q-1}_k$ and $x\in \varphi_{q-1}(\tau)^{(0)}$ we have $\ell(x)\le L$. Then for every $\sigma\in (\{1_G\}\times G^q)\cap\Delta^q_k$ and $y\in \varphi_{q-1}\circ\partial_q(\sigma)^{(0)}$ we have $\ell(y)\le L+k$. So $\mathcal M\subseteq \prod \mathcal X^{L+k}$ as a subset of a compact set has compact closure $\bar{\mathcal M}$. 

If $\mu\in \partial\mathcal M$, a point on the boundary of $\mathcal M$, then there exists some $\mu'\in \mathcal M$ with $\mu\in \prod_{t\in T^{(0)}}\mu'(t)U$. If $\mu'\in C_{S,c}$, then $\mu\in C_{S,c}U=U_{S,c}$. Since by condition 2 we have that $\tilde \hlgy_{q-1}(\VR{k_{q-1}+1}{G_{\chi}})$ vanishes in $\tilde \hlgy_{q-1}(\VR{k_q}{G_{\chi}})$ each $\mu'\in M$ is contained in some $C_{S,c}$ and by Lemma~\ref{lem:connectedshape} we can assume $S$ to be of the form described in the above list. Thus, $U_{S,c}$ over all $S,c$ are an open cover of $\bar{\mathcal M}$. Since $\bar {\mathcal M}$ was shown to be compact, there exist $S_1,c_1,\ldots, S_n,c_n$ with
\[
 U_{S_1,c_1}\cup\cdots \cup U_{S_n,c_n}\supseteq \bar {\mathcal M}.
\]
So if $\sigma\in (\{1_G\}\times G^q)\cap \Delta^q_k$, then there exists $i\in\{1,\ldots,n\}$ with $\varphi_{q-1}\circ\partial_q(\sigma)\in U_{S_i,c_i}$. Define $\varphi_q(\sigma)$ to be the chain of shape $S_i$ with vertices
\begin{align*}
 \varphi_q(\sigma)^{(0)}:S_i^{(0)}&\to G\\
 s&\mapsto \begin{cases}
            \varphi_{q-1}\circ\partial_q(\sigma)^{(0)}(s) & s\in T^{(0)}\\
            c_i(s) & s\in R_i^{(0)}.
           \end{cases}
\end{align*}
Then
\begin{align*}
 v(\varphi_q(\sigma))-v(\sigma)
 &\ge v(\varphi_{q-1}\circ \partial_q(\sigma))-K-v(\sigma)\\
 &=v(\varphi_{q-1}\circ\partial_q(\sigma))-v(\partial_q (\sigma))-K\\
 &\ge (m+1-(q-1))K-K\\
 &=(m+1-q)K.
\end{align*}
Since $(\{1_G\}\times G^q)\cap \Delta^q_k$ is a $\ZZ G$-basis for $\chains q k G$, we have defined 
\[
\varphi_q:\chains q k G\to \chains q {k_q+1} G
\]
as a $\ZZ G$-chain endomorphism on the $q$-skeleton that is finitely modeled and raises valuation. In this way, we have constructed the $\varphi$ of condition 3.

We now show that condition 3 implies condition 1: Suppose condition 3, there exists a chain endomorphism $\varphi^{\tilde K}$ of $\chains * {\tilde K} G$ (where $\tilde K$ is a choice of $k,l_1,l_2$ to be defined later, where we assume that $\varphi^{l_2},\varphi^{l_1},\varphi^k$ match when restricted and denote all of them by $\varphi$) that raises the $\chi$-value and suppose we did already show that $\chi\in \TopS^{m-1}(G;\ZZ)$. 

Let $k\ge 0$ be a large enough index. Since $G$ is of type $\Clg m$ there exists $l_1\ge k$ such that $\tilde \hlgy_{m-1}(\VR k G)$ vanishes in $\tilde \hlgy_{m-1}(\VR{l_1}G)$. By Lemma~\ref{lem:homotopy_homology_locallycompact} there exists a chain homotopy $\eta:\chains {m-1} k G\to \chains m {l_2} G$ joining $\varphi$ to the identity with 
    \[
    \eta_m(\tau)^{(0)}\subseteq \bigcup_{q=0}^m\bigcup_{\substack{\sigma\le \tau,\\ [\tau:\sigma]=m-q}} (\varphi_q(\sigma))^{(0)}\cup \tau^{(0)}.
   \]
   for every $\tau\in \Delta^m_k$. Since $\varphi$ raises $\chi$-value and $\id$ does not lower $\chi$-value, we obtain
\[
\inf_{\sigma=(1_G,x_1,\ldots,x_{m-1})\in\Delta^{m-1}_k}(v(\eta_{m-1}(\sigma))-v(\sigma))\ge 0.
\]
 Then we define
 \[
  L:=\inf_{\sigma=(1_G\times G^m)\cap \Delta^m_{l_1}}(v(\varphi(\sigma))-v(\sigma)>0.
 \]
This value exists by the assumption on $\varphi_m$. We now show that $\tilde \hlgy_{m-1}(\VR k {G_\chi})$ vanishes in $\tilde \hlgy_{m-1}(\VR{\max(l_1,l_2)}{G_\chi})$. To that end, let $z\in \chains {m-1} k {G_{\chi}}$ be a cycle. Then $z-\varphi_{m-1}(z)=\partial_m\eta_{m-1}(z)$. Define
\[
 c_0:=\eta_{m-1}(z),\quad c_i:=\varphi^{\circ i}_m(c_0),\quad z_i:=\varphi^{\circ i}_{m-1}(z).
\]
Then
\[
 \partial_m c_i=\varphi^{\circ i}_{m-1}\circ \partial_m c_0=\varphi^{\circ i}_{m-1}(z-\varphi_{m-1}(z))=z_i-z_{i+1}.
\]
 Now there exists $c\in \chains m {l_1} G$ with $\partial_m c=z$. Then $(n+1)L\ge -v(c)$ for some $n\in \NN$. Define 
\[
 \tilde c:=\sum_{i=0}^n c_i+\varphi_m^{\circ n+1}(c).
\]
Then
\[
 \partial_m\tilde c=\sum_{i=0}^n\partial_m c_i+\varphi_{m-1}^{\circ n+1}(\partial_m c)=\sum_{i=0}^n(z_i-z_{i+1})+z_{n+1}=z
\]
and 
\begin{align*}
 v(\tilde c)
 &\ge \min_i(v(c_i),\varphi^{\circ n+1}_m(c))\\
 &\ge \min(0,(n+1)L+v(c))\\
 &\ge \min(0,0)\\
 &=0.
\end{align*}
Thus, $\tilde \hlgy_{m-1}(\VR k{G_\chi})$ vanishes in $\tilde \hlgy_{m-1}(\VR{\max(l_1,l_2)}{G_\chi})$. In this way, we showed $\chi\in \TopS^m(G,\ZZ)$. That is condition 1.
\end{proof}

\begin{thm}
 \label{thm:crit2}
 Let $m\in\NN$, let $G$ be a locally compact Hausdorff group with homological connecting vector $(k_0,\ldots,k_m)$ and set $k:=k_m+1$. If $\chi:G\to \RR$ is a nonzero character then the following are equivalent:
 \begin{enumerate}
  \item $\chi\in \TopS^m(G;\ZZ)$;
  \item there exist $K>0$ and a finitely modeled chain endomorphism of $\ZZ G$-complexes $\varphi_*:\chains * k G\to \chains * k G$ extending the identity on $\ZZ$ such that for every $q=0,\ldots,m$ and $c\in \chains q k G$ we have
  \[
   v(\varphi_q(c))-v(c)\ge K.
  \]
 \end{enumerate}
\end{thm}
\begin{proof}
 We first show that condition 1 implies condition 2. Suppose condition 1 that $\chi\in \TopS^m(G;\ZZ)$. Then Theorem~\ref{thm:sigmacrit_lc_homol} implies there exists a number $l\ge 0$ and for every $n\ge l$ a finitely modeled chain endomorphism $\varphi_{n,*}:{\chains * n G}^{(m)}\to {\chains * l G}^{(m)}$ extending the identity on $\ZZ$ and raising valuation by a number $K$. 
 
 If $k\ge l$ denote by $\iota_{lk}$ the inclusion $\chains * l G \subseteq \chains * k G$. Then $\iota_{lk}\circ \varphi_{k,*}$ is the desired finitely modeled chain endomorphism. 
 
 If conversely $k<l$, then since $(k_0,\ldots,k_m)$ is a homological connecting vector for $G$, Theorem~\ref{thm:mu} implies for every $n\ge k$ there exists a finitely modeled chain endomorphism $\mu_{n,*}:{\chains * n G}^{(m)}\to {\chains * k G}^{(m)}$ extending the identity on $\ZZ$. Then
\begin{align*}
 \inf_{\substack{c\in\chains q l G\\q=0,\ldots,m}} (v(\mu_q(c))-v(c))
 &\ge \min_{\substack{\sigma\in \Delta^q_l\cap (1_G\times G^q)\\q=0,\ldots,m}}(v(\mu_q(\sigma))-v(\sigma))\\
 &=:L\in \mathbb R
\end{align*}
Then $i\cdot K\ge -L$ for some $i\in \mathbb N$. Then denote by $\iota_{kl}$ the inclusion $\chains * k G \subseteq \chains * l G$. Then $\mu_{l,*}\circ \varphi_{l,*}^{\circ i+1}\circ \iota_{kl}$ is the desired finitely modeled chain endomorphism:
\begin{align*}
 v(\mu_q\circ \varphi_{l,q}^{\circ i+1}\circ \iota_{kl}(c))
 &\ge L+v(\varphi_{l,q}^{\circ i+1}\circ \iota_{kl}(c))\\
 &\ge L+(i+1)K+v(\iota_{kl}(c))\\
 &\ge K+v(c).
\end{align*}

We now assume condition 2 and show condition 1. Let $n\ge k$ be a number. Since $(k_0,\ldots,k_m)$ is a homological connecting vector for $G$, there exists a finitely modeled chain endomorphism $\mu_*:{\chains * n G}^{(m)}\to {\chains * k G}^{(m)}$ extending the identity on $\ZZ$. Condition 2 also provides us with a finitely modeled chain endomorphism $\varphi_*:{\chains * k G}^{(m)}\to {\chains * k G}^{(m)}$ extending the identity on $\ZZ$ with $v(\varphi(c))-v(c)\ge K$ for every $c\in \chains q k G$ and $q=0,\ldots,m$. Then, as before, there exists some $L\in \RR$ such that for every $q=0,\ldots,m$ and $x=(1,g_1,\ldots,g_q)\in\Delta^q_n$ we have
 \[
  v(\mu(x))-v(x)\ge L
 \]
Then $i\cdot K\ge -L$ for some $i\in \NN$. 

Then $\varphi_*^{\circ i+1}\circ \mu_*:{\chains * n G}^{(m)}\to {\chains * k G}^{(m)}$ is a chain endomorphism extending the identity on $\ZZ$ with
\[
 v(\varphi_q^{\circ i+1}\circ \mu_q(c))\ge (i+1)K+v(\mu_q(c))\ge (i+1)K+L+v(c)\ge K+v(c)
\]
 By Theorem~\ref{thm:sigmacrit_lc_homol} this proves that $\chi\in \TopS^m(G;\ZZ)$.
\end{proof}

\section{Stability of the Sigma-invariants}
\label{sec:open}
This section proves Theorem~\ref{thma:open}.

Similarly to \cite{Hartmann2024c}, we denote by $\TopHom(G,\RR)$ the set of continuous group homomorphisms $G\to \RR$ and endow it with the compact-open topology. Namely, a sub-base for open sets are the sets of the form
\[
 \mathcal U(K,V):=\{\chi:G\to \RR\mid \chi(K)\subseteq V\}
\]
for every compact set $K$ in $G$ and open set $V$ in $\RR$. If $\mathcal X\subseteq G$ is a compact generating set for $G$, we define a mapping 
\begin{align*}
 u:\TopHom(G,\RR)&\to \RR^{|\mathcal X|}\\
 \chi&\mapsto (\chi(x))_{x\in \mathcal X}
\end{align*}
where $\RR^{|\mathcal X|}$, as a (possibly) infinite product of copies of $\RR$, is endowed with the product topology. We also denote by $w_{\mathcal X}:\F(\mathcal X)\to \bigoplus_{\mathcal X}\ZZ$ the abelianization of the free group $\F(\mathcal X)$ on $\mathcal X$. And if $t:G\to \F(\mathcal X)$ is a transversal of the usual projection, we define $w:=w_{\mathcal X}\circ t$. Then 
\[
 \chi(g)=\sum_{x_i\in t(g)}\chi(x_i)=\sum_{w(g)(x)\not=0}w(g)(x)\chi(x)=:\langle w(g),u(\chi)\rangle
\]
does not depend on the representation $t$ or the choice of generating set $\mathcal X$. A vector $y\in \RR^{|\mathcal X|}$ is in the image of $u$ if for every word $r$ in the defining relations 
\[
 \langle w_{\mathcal X}(r),y\rangle=0.
\]
So $u(\TopHom(G,\RR))$ is a closed linear subspace of $\RR^{|\mathcal X|}$. We show that $\Sigma^m(G,\ZZ)$ is a cone over an open subset of $\TopHom(G,\RR)\setminus\{0\}$. If $c\in \ZZ[G^{q+1}]$ denote by $c^{(0)}$ the vertices in $G$ appearing in the chain. Let $\varphi_q:{\chains q k G}^{(m)}\to {\chains q k G}^{(m)}$ be a finitely modeled chain endomorphism extending the identity on $\ZZ$ (one obtains such a chain endomorphism in the proof of Theorem~\ref{thm:crit2}). Then define a mapping 
 \begin{align*}
 u_\varphi:\TopHom(G,\RR)&\to \RR\\
 \chi&\mapsto \min_{\substack{\bar x\in(\{1_G\}\times G^q)\cap 
\Delta^q_k,\\q=0,\ldots,m}}\left(\min_{g\in \varphi_q(\bar x)^{(0)}}\langle 
w(g),u(\chi)\rangle-\min_{g\in \bar x^{(0)}}\langle w(g),u(\chi)\rangle\right).
\end{align*}

\begin{lem}
\label{lem:uvarphi_continuous}
 If $\TopHom(G,\RR)$ is equipped with the compact-open topology, then $u_\varphi$ is continuous.
\end{lem}
\begin{proof}
 If $q\in\{0,\ldots,m\}$, we define $u_\varphi^q(\chi):=\min_{\bar x\in (1_G\times G^q)\cap \Delta^q_k}(v(\varphi_q(\bar x))-v(\bar x))$. Then
 \[
  u_\varphi=\min(u_\varphi^0,\ldots,u_\varphi^m).
 \]
So, it remains to show that each $u_\varphi^q$ is continuous. If $c$ is a $q$-chain, define $\Int c^{(0)}$ to be the set of $G$-values of vertices in $c^{(0)}\setminus \partial c^{(0)}$. Then
 \begin{align*}
  u_\varphi^q(\chi)
  &=\min_{\bar x\in (1_G\times G^q)\cap \Delta^q_k}(\min\left(\min_{g\in \Int\varphi^{(0)}(\bar x)}\chi(g),\min_{g\in \partial\varphi(\bar x)^{(0)}}\chi(g)\right)-v(\bar x))\\
  &=\min(\min_{\bar x}\left(\min_{g\in \Int\varphi^{(0)}(\bar x)}\chi(g)-v(\bar x)\right),\min_{\bar x}\left(\min_{g\in \partial\varphi(\bar x)^{(0)}}\chi(g)-v(\bar x)\right))\\
  &=\min(\min_{\bar x}\left(\min_{g\in \Int\varphi^{(0)}(\bar x)}\chi(g)-v(\bar 
x)\right),u^{q-1}_\varphi(\chi)).
 \end{align*}
Denote $t:=\varphi_0(1_G)$. Note that $u_\varphi^0(\chi)=\chi(t)$ as a function of $\chi$ is continuous: If $\chi\in \TopHom(G,\RR)$ and $\varepsilon>0$ then $U(\{t\},(\chi(t)-\varepsilon,\chi(t)+\varepsilon))$ maps to $(\chi(t)-\varepsilon,\chi(t)+\varepsilon)$. It remains to show
\[
 \Int u^q_\varphi(\chi):=\min_{\bar x}(\min_{g\in \Int\varphi^{(0)}(\bar x)}\chi(g)-v(\bar x))
\]
is continuous as a function of $\chi$. Now $(1_G\times G^q)\cap \Delta^q_k$ decomposes as a finite union $\mathcal C_1\cup\cdots\cup \mathcal C_n$, such that for each $i\{1,\ldots,n\}$ the set $D_i:=\Int \varphi(\bar x)^{(0)}$ is constant as $\bar x$ varies in $\mathcal C_i$. Define $C_i:=\{g\in {\bar x}^{(0)}\mid \bar x\in \mathcal C_i\}$. 

If $\chi\in \TopHom(G,\RR)$ and $\varepsilon>0$, then $V:=(\Int u_\varphi^q(\chi)-\varepsilon,\Int u_\varphi^q(\chi)+\varepsilon)$ is an open set in $\RR$. We provide an open $U$ containing $\chi$ that $u_\varphi^q(\chi)$ maps to $V$. 

Note that $D_i$ contains finitely many points. The same cannot be said about $C_i$, however, $C_i$ is relatively compact and therefore the cover $(\chi^{-1}(-\tfrac \varepsilon 8+r,r+\tfrac \varepsilon 8))_{r\in\RR}$ of $\bar C_i$ has a finite subcover.
\[
\chi^{-1}(-\tfrac \varepsilon 8+r_{i1},r_{i1}+\tfrac \varepsilon 8),\ldots, \chi^{-1}(-\tfrac \varepsilon 8+r_{ik},r_{ik}+\tfrac \varepsilon 8).
\]
For each $j\in \{1,\ldots,k\}$, we fix an element $x_{ij}$ in the fiber $\chi^{-1}(r_{ij})$. Then $C_i'$ is the collection of the $x_{ij}$, a finite set of points.

Then define
\begin{align*}
 U:=&\bigcap_{i=1}^n\bigcap_{x\in C_i'} \mathcal U(\bar C_i\cap\chi^{-1}[-\tfrac\varepsilon 8+\chi(x),\chi(x)+\tfrac\varepsilon 8],(\chi(x)-\tfrac\varepsilon 4,\chi(x)+\tfrac \varepsilon 4))\\
 &\cap \bigcap_{i=1}^n\bigcap_{y\in D_i} \mathcal U(\{y\},(\chi(y)-\tfrac \varepsilon 2,\chi(y)+\tfrac \varepsilon 2)).
\end{align*}
Define 
\[
 d_i:=\min_{g\in D_i} \chi(g)\quad\mbox{and}\quad c_i:=\inf_{g\in C_i}\chi(g).
\]

If $\chi'\in U$ and $x\in C_i$ then there exists some $x_{ij}\in C_i'$ such that $|\chi(x)-\chi(x_{ij})|\le \tfrac \varepsilon 8$ then
\[
 |\chi(x)-\chi'(x)|\le |\chi(x)-\chi(x_{ij})|+|\chi(x_{ij})-\chi'(x)|<\tfrac \varepsilon 8+\tfrac \varepsilon 4<\tfrac \varepsilon 2.
\]
This can be used in the following computation:
\begin{align*}
  \Int u^q_\varphi(\chi')
  &=\min_i(\min_{g\in D_i}\chi'(g)-\min_{g\in C_i}\chi'(g))\\
  &\in\min\{(d_i-\tfrac\varepsilon 2,d_i+\tfrac \varepsilon 2)-(c_i-\tfrac \varepsilon 2,c_i+\tfrac \varepsilon 2)\mid i=1,\ldots,n\}\\
  &=\min \{(d_i-c_i-\varepsilon,d_i-c_i+\varepsilon)\mid i=1,\ldots,n\}\\
  &=(\Int u_\varphi^q(\chi)-\varepsilon,\Int u_\varphi^q(\chi)+\varepsilon)\\
  &=V.\qedhere
 \end{align*}

\end{proof}

By Lemma~\ref{lem:uvarphi_continuous} the set $U(\varphi):=u_\varphi^{-1}(0,\infty)$ is open in $\TopHom(G,\RR)$.

\begin{prop}
 Let $G$ be a group with homological connecting vector $(k_0,\ldots,k_m)$. If 
for $k:=k_m+1$ there is a finitely modeled $\ZZ G$-chain 
endomorphism $\varphi_q:\chains q k G\to \chains q k G$ extending the identity 
on $\ZZ$, then $U(\varphi)\subseteq \Sigma^m(G,\ZZ)$.
\end{prop}
\begin{proof}
 Suppose $\chi\in U(\varphi)$. Then
 \begin{align*}
 0
 &<\min_{\substack{\bar x=(1_G,g_1,\ldots,g_q)\in \Delta^q_k,\\ q=0,\ldots,m}}\left(\min_{g\in \varphi_q(\bar x)^{(0)}}\langle w(g),v(\chi)\rangle-\min_{g\in \bar x^{(0)}}\langle w(g),v(\chi)\rangle\right)\\
 &=\min_{\substack{\bar x=(1_G,g_1,\ldots,g_q)\in \Delta^q_k,\\ q=0,\ldots,m}}\left(\min_{g\in \varphi_q(\bar x)^{(0)}}\chi(g)-\min_{g\in \bar x^{(0)}}\chi(g)\right)\\
 &=\min_{\substack{\bar x=(1_G,g_1,\ldots,g_q)\in 
\Delta^q_k,\\ q=0,\ldots,m}}\left(v(\varphi_q(\bar x))-v(\bar x)\right)=:L.
 \end{align*}
 Let $l\ge k$ be a number. Since $(k_0,\ldots,k_m)$ is a homological connecting vector for $G$ there exists a finitely modeled chain endomorphism $\mu_*:{\chains * l G}^{(m)}\to 
{\chains * k G}^{(m)}$ extending the identity on $\ZZ$. Then for some $K\in \RR$ for every $x=(1,g_1,\ldots,g_q)\in\Delta^q_l,q=0,\ldots,m$ we have
 \[
  v(\mu(x))-v(x)\ge K
 \]
 as a consequence of Lemma~\ref{lem:homotopy_homology_locallycompact}. Then $nL\ge -K$ for some $n\in \NN$. Then $\varphi_q^{\circ n+1}\circ \mu_q:\chains q l G\to \chains q k G$ is a chain endomorphism extending the identity on $\ZZ$ with
\[
 v(\varphi_q^{\circ n+1}\circ \mu_q(c))\ge (n+1)L+v(\mu_q(c))\ge (n+1)L+K+v(c)\ge L+v(c)
\]
and $\mbox{im }\varphi_q^{\circ n+1}\circ \mu_q \subseteq \chains q k G$ for 
every $q=0,\ldots,m$. By Theorem~\ref{thm:sigmacrit_lc_homol} we obtain that the 
character $\chi\in \TopS^m(G;\ZZ)$.
\end{proof}

\begin{thm}
\label{thm:open}
 The subset $\TopS^m(G,\ZZ)$ is a cone over an open set in $\TopHom(G,\RR)$ provided $\Sigma^m\not=\emptyset$.
\end{thm}
\begin{proof}
Note that $\{0\}$ is the cone over an empty set, and thus the case $\Sigma^m=\{0\}$ is valid. 

 If $\chi\not=0$ is a character on $G$ that belongs to $\Sigma^m$ then there exists some chain endomorphism $\varphi$ that is a witness for $\chi$ belonging to $\Sigma^m$. Then $U(\varphi)$ is an open neighborhood of $\chi$ in $\Sigma^m$. So $\Sigma^m\setminus\{0\}$ is open. Since for every $\lambda>0$ we have $\lambda\chi\in\Sigma^m$ if and only if $\chi\in\Sigma^m$ and also $0\in \Sigma^m$ if $\Sigma^m\not=\emptyset$ the set $\Sigma^m$ is a cone.
\end{proof}

%% file: ge1.tex
\section{Group extensions by kernels of type {$\Clg m$}}
This section proves Theorem~\ref{thma:ge1}.

\begin{rem}
\label{rem:kernel+cokernel}
 The category of locally compact Hausdorff groups and continuous group homomorphisms with closed image admits kernels and cokernels. If $\alpha:G\to H$ is a continuous group homomorphism with closed image between locally compact Hausdorff groups, then the kernel of $\alpha$ is the kernel of $\alpha$ as a group homomorphism, such that $\ker \alpha\to G$ is a closed embedding (Since $H$ is Hausdorff, the point $1_H$ is closed and since $\alpha$ is continous, $\ker\alpha=\alpha^{-1}(1_H)$ is closed). The cokernel of $\alpha$ is the cokernel of $\alpha$ as a group homomorphism, such that $H\to \mbox{coker } \alpha$ is a quotient map. Since $\im \alpha$ is closed the space $\coker\alpha$ is again locally compact Hausdorff.
\end{rem}

We now study a short exact sequence of topological groups $1\to N\to G\to Q\to 1$. That means $N$ is the kernel of $G\to Q$ and $Q$ is the cokernel of $N\to G$. Let $\chi:G\to \RR$ be a character that vanishes on $N$. Denote by $\pi:G\to G/N=:Q$ the projection to the factor group. Since $\chi$ vanishes on $N$ there is a character $\bar \chi$ on $Q$ with $\bar\chi\circ\pi=\chi$.

\begin{rem}
\label{rem:closed_subgroup}
 If $N$ is a closed subgroup of $G$ and $E\subseteq G\times G$ an entourage, then $\{x^{-1}y\mid (x,y)\in E\}$ is contained in a compact set $C\subseteq G$. Then $\{x^{-1}y\mid (x,y)\in E\cap N\times N\}$ is contained in the set $C\cap N$, which as a closed subset of a compact set $C$ is compact itself. Thus, $E\cap N\times N$ is an entourage in $N$. If $G$ is locally compact and Hausdorff, so is $N$. This means that the subspace metric from $G$ on $N$ and the metric assigned to a compact generating set of $N$, if it exists, induce the same coarse structure on $N$.
\end{rem}

\begin{lem}
\label{lem:coarse=easy}
 If $k\ge 0$ and $m\in \NN$, then 
 \[
 \Delta_k[N]:=\{x\in G\mid d(x,y)\le k,\mbox{ for some }y\in N\}
\]
equipped with the subspace metric from $G$ is of type $\Flg m$, if and only if $N$ is of type $\Clg m$.
\end{lem}
\begin{proof}
 By Remark~\ref{rem:kernel+cokernel} the kernel $N$ appears as a closed subgroup of $G$. And by Remark~\ref{rem:closed_subgroup} the subgroup $N$ is of type $\Clg m$ if and only if the subspace $(N,d)$ of $(G,d)$ is of type $\Flg m$. 
 
 Denote by $\iota:N\to \Delta_k[N]$ the inclusion and by $\rho:\Delta_k[N]\to N$ the map sending each $x\in \Delta_k[N]$ to a point $y\in N$ within $k$-distance of $x$. If $x_1,x_2\in \Delta_k[N]$ have $d(x_1,x_2)\le l$ then 
 \[
  d(\rho(x_1),\rho(x_2))\le k+d(x_1,x_2)+\le 2k+l.
 \]
So $\rho$ defines a coarsely Lipschitz map. Obviously $\iota$ is also coarsely Lipschitz. Then $\rho\circ \iota$ is $k$-close to the identity on $N$ and $\iota\circ\rho$ is $k$-close to the identity on $\Delta_k[N]$. As a result, we find that $\iota$ and $\rho$ are inverses in the coarse category. This means that $N$ and $\Delta_k[N]$ both with the subspace metric of $G$ are coarsely isomorphic. That implies $(N,d)$ is of type $\Flg m$ if and only if $(\Delta_k[N],d)$ is of type $\Flg m$.
\end{proof}

\begin{thm}
\label{thm:ge1}
 If $1\to N\to G\to Q\to 1$ is a short exact sequence of locally compact groups then, for every $m\ge 1$: 
 \begin{enumerate}
  \item if $N$ is of type $\Clg m$ and $\bar\chi\in \TopS^m(Q;\ZZ)$, then $\chi\in \TopS^m(G;\ZZ)$;
  \item if $N$ is of type $\Clg {m-1}$ and $\chi\in \TopS^m(G;\ZZ)$, then $\bar\chi\in \TopS^m(Q;\ZZ)$.
 \end{enumerate}
\end{thm}
\begin{proof}
Suppose $\mathcal X$ is a compact symmetric generating set for $G$ containing the identity. Then $\pi(\mathcal X)$ is a compact generating set for $Q$. The normal subgroup $N$ is endowed with the subspace metric from $G$.

Firstly, we show claim 1. Suppose $N$ is of type $\Clg m$ and $\bar\chi\in \TopS^m(Q;\ZZ)$.
The outline of this part of the proof is as follows: First, we construct for some $n\ge 0$ for every $k\ge 0$ a chain homomorphism of $\ZZ$-modules $\varphi_*:{\chains * k {Q_{\bar \chi}}}^{(m)}\to {\chains * n {G_\chi}}^{(m)}$ using that $N$ is of type $\Clg m$. Then we construct a chain homotopy that joins $\varphi\circ\pi:{\chains * k {G_\chi}}^{(m)}\to{\chains * n {G_\chi}}^{(m)}$ to the identity on ${\chains * k {G_\chi}}^{(m)}$ also by just using that $N$ is of type $\Clg m$. Then we use $\bar\chi\in \TopS^m(Q;\ZZ)$ to prove that $\varphi\circ \pi$ sends cycles to boundaries. 

  We construct a sequence of numbers $n_0\le n_1\le n_2\le\cdots\le n_m$ and for every $k\ge 0$  a chain homomorphism $\varphi_*:{\chains * k {Q_{\bar\chi}}}^{(m)}\to \ZZ[{G_\chi}^{*+1}]^{(m)}$ extending the identity on $\ZZ$ with $\im \varphi_q\subseteq \chains q {n_q} {G_\chi}$ for $q=0,\ldots,m$. Suppose $t:Q\to G$ is a transversal to $\pi$ (by which we mean a set-theoretic section). As a shorthand, we write $t(x)=:\tilde x$ for each $x\in Q_{\bar \chi}$. And as a shorthand for ${\tilde x_0}^{-1}\varphi_q(x_0,\ldots,x_q)$, we write $w_{0\cdots q}$. If $q=0$ the map $\varphi_0$ is defined to send $x$ to $\tilde x$. Obviously $w_0=1_G$. We now construct the chain map for $q=1$. If $(x_0,x_1)\in \chains 1 k {Q_{\bar \chi}}$, then
 \[
  {\tilde x_0}^{-1}\tilde x_1\in \mathcal X^kN=N\mathcal X^k=\Delta_k[N].
 \]
Since $N$ is of type $\Clg 1$ and Lemma~\ref{lem:coarse=easy}, we have $\homology 0 {n_1} {\Delta_k[N]}=\ZZ$ for some $n_1$. Thus, there exists $w_{01}\in \chains 1 {n_1} {\Delta_k[N]}$ with $\diff 1 w_{01}={\tilde x_0}^{-1}\tilde x_1-1$. Define $\varphi_1(x_0,x_1):=\tilde x_0w_{01}$. Then obviously 
\[
 \partial_1\circ \varphi_1(x_0,x_1)=\tilde x_0\partial_1 w_{01}=\tilde x_0({\tilde x_0}^{-1}\tilde x_1-1=\tilde x_1-\tilde x_0=\varphi_0\circ \partial_1(x_0,x_1).
\]
For $q>1$, suppose that $\varphi_0,\ldots,\varphi_{q-1}$ have been constructed. Part of the induction hypothesis is that 
\[
 w_{0\cdots \hat i \cdots q}\in \chains {q-1} {n_{q-1}}{\Delta_{(q-1)k}[N]}
\]
which is true for $q-1=1$. Let $(x_0,\ldots,x_q)\in \chains q k {Q_{\bar \chi}}$ be a simplex. Then
\[
 {\tilde x_0}^{-1}\varphi_{q-1}\circ \diff q(x_0,\ldots,x_q)={\tilde x_0}^{-1}(\tilde x_0w_{0\cdots 
q-1}-\tilde x_0w_{0\cdots q-2,q}+\cdots\pm \tilde x_1 w_{1\cdots q}).
\]
From the induction hypothesis, we know that $w_{1\cdots q}\in \chains {q-1}{n_{q-1}}{\Delta_{(q-1)k}[N]}$. If $y$ is a vertex of the chain $w_{1\cdots q}$, then
\[
 \tilde x_0^{-1}\tilde x_1y\in \mathcal X^k NN\mathcal X^{(q-1)k}=N\mathcal X^{qk}=\Delta_{qk}[N].
\]
Thus, 
\[
{\tilde x_0}^{-1}(\tilde x_0w_{0\cdots q-1}-\tilde x_0w_{0\cdots q-2,q}+\cdots\pm \tilde x_1 w_{1\cdots 
q})\in \chains {q-1} {n_{q-1}} {\Delta_{qk}[N]}.
\]
Since $N$ is of type $\Clg q$ and by Lemma~\ref{lem:coarse=easy} there exists some $n_q\ge n_{q-1}$ such that $\homology {q-1}{n_{q-1}}{\Delta_{qk}[n]}$ vanishes in $\homology {q-1}{n_q}{\Delta_{qk}[n]}$. Then there is $w_{0\cdots q}\in \chains q {n_q}{\Delta_{qk}[N]}$ with 
\[
 \diff q w_{0\cdots q}={\tilde x_0}^{-1}(\tilde x_0w_{0\cdots q-1}-\tilde x_0w_{0\cdots 
q-2,q}+\cdots\pm \tilde x_1 w_{1\cdots q}).
\]
Define $\varphi_q(x_0,\ldots, x_q):=\tilde x_0w_{0\cdots q}$. Then obviously 
\[
 \partial_q\circ \varphi_q(x_0,\ldots, x_q)=\tilde x_0 \partial_q w_{0\cdots q}=\tilde x_0w_{0\cdots q-1}-\tilde x_0w_{0\cdots 
q-2,q}+\cdots\pm \tilde x_1 w_{1\cdots q}=\varphi_{q-1}\circ\partial_q(x_0,\ldots,x_q).
\]

Now we construct $m_0\le m_1\le \cdots\le m_{m-1}$ and for every $k\ge 0$ a chain homotopy of $\ZZ$-modules $\lambda_*:{\chains q k {G_\chi}}^{(m-1)}\to {\ZZ[G_\chi^{*+1}]}^{(m)}$ joining $\varphi\circ\pi_*$ to $\id$ with $\im \lambda_q\subseteq \chains {q+1}{m_q} {G_\chi}$ for each $q=0,\ldots,m-1$. As a shorthand for $y_0^{-1}\lambda_q(y_0,\ldots,y_q)$, we write $v_{0\cdots q}$. Define $\lambda_{-1}=0$. For $q>-1$, suppose that $\lambda_{-1},\ldots,\lambda_{q-1}$ have been constructed. Part of the induction hypothesis is that 
\[
 v_{0\cdots q-1}\in \chains q {m_{q-1}} {\Delta_{k(q-1)}[N]}.
\]
This is of course true for $q-1=-1$. If $(y_0,\ldots,y_q)\in \chains q k {G_\chi}$, then
\begin{align}
\label{eq:ge101}
 \varphi\circ \pi_*(y_0,\ldots,y_q)\in t\circ\pi(y_0)\cdot \chains q {n_q} 
{\Delta_{qk}[N]}=y_0\chains q {n_q} {\Delta_{qk}[N]},
\end{align}
since $y_0^{-1}\cdot t\circ\pi(y_0)\in N$. Also,
\begin{align}
\label{eq:ge102}
 \id(y_0,\ldots,y_q)\in y_0\chains q k {\Delta_k[1_G]},
\end{align}
where $\Delta_k[1_G]$ denotes (as the notation suggests) the $k$-ball around $1_G$. By induction hypothesis 
\begin{align*}
y_0^{-1}\lambda_{q-1}(y_1,\ldots,y_q)
&=y_0^{-1}y_1v_{1\cdots q}\\
&\in y_0^{-1}y_1 \chains q {m_{q-1}} {\Delta_{k(q-1)}[N]}\\
&\subseteq \chains q {m_{q-1}}{\Delta_{kq}[N]}.
\end{align*}
Thus,
\begin{align}
\label{eq:ge103}
\begin{split}
 y_0^{-1}\cdot \lambda_{q-1}\circ\diff q(y_0,\ldots,y_q)
 &=y_0^{-1}(y_0v_{0\cdots q-1}-y_0v_{0\cdots q-2,q}+\cdots \pm y_1 v_{1\cdots q})\\
 &\in \chains q {m_{q-1}}{\Delta_{kq}[N]}
 \end{split}
\end{align}
Suppose without loss of generality $m_{q-1}\ge n_q$, then equations~\ref{eq:ge101},\ref{eq:ge102},\ref{eq:ge103} combine to
\[
 y_0^{-1}(\varphi\circ\pi-\id-\lambda_{q-1}\circ\partial_q)(y_0,\ldots,y_q)\in \chains 
q {m_{q-1}} {\Delta_{qk}[N]}.
\]
Now $\partial_q\circ (\varphi\circ\pi-\id-\lambda_{q-1}\circ\partial_q)=0$ by basic calculation. Since $N$ is of type $\Clg q$ and Lemma~\ref{lem:coarse=easy}, there exists some $m_q\ge m_{q-1}$ such that $\homology q {m_{q-1}}{\Delta_{qk}[N]}$ vanishes in $\homology q {m_q} {\Delta_{qk}[N]}$. Then there exists some $v_{0\cdots q}\in \chains {q+1} {m_q} {\Delta_{qk}[N]}$ with
\[
 \diff {q+1} v_{0\cdots 
q}=y_0^{-1}(\varphi\circ\pi-\id-\lambda_{q-1}\circ\partial_q)(y_0,\ldots,y_q).
\]
Define $\lambda_q(y_0,\ldots,y_q):=y_0 v_{0\cdots q}$. This way
\begin{align*}
 (\partial_{q+1}\circ \lambda_q+\lambda_{q-1}\circ\partial_q)(y_0,\ldots,y_q)
 =&y_0 \partial_{q+1} v_{0\cdots q}+\lambda_{q-1}\circ \partial_q(y_0,\ldots,y_q)\\
 =&y_0\cdot y_0^{-1}(\varphi\circ\pi-\id-\lambda_{q-1}\circ\partial_q)(y_0,\ldots,y_q)\\
 &+\lambda_{q-1}\circ \partial_q(y_0,\ldots,y_q)\\
 =&(\varphi\circ\pi=\id)(y_0,\ldots,y_q).
\end{align*}
So we constructed a chain homotopy joining $\varphi\circ \pi$ to the identity.

If $z\in \chains {m-1} k {G_\chi}$ is a cycle, then $\pi(z)\in \chains {m-1} k {Q_{\bar \chi}}$ is also a cycle. Since $\homology {m-1} k {Q_{\bar \chi}}$ is essentially trivial, there exists $c\in \chains m l {Q_{\bar \chi}}$ with $\diff m c =\pi(z)$. Then
\[
 \diff m (\varphi_m(c)-\lambda_{m-1}(z))=\partial_m\circ \varphi_m(c)-\partial_m\circ\lambda_{m-1}(z)=\varphi\circ\pi(z)-(\varphi\circ\pi(z)-z)=z.
\]
Thus, $\homology q k {G_\chi}$ is trivial as an ind-object. Thus, $\chi\in \TopS(G;\ZZ)$. 

Now suppose $N$ is of type $\Clg {m-1}$ and $\chi\in \TopS^m(G;\ZZ)$. The outline of this part of the proof is as follows. First, we reuse the chain homomorphism $\varphi$ constructed in the first part of the proof, this time for $m-1$ instead of $m$ since we only require that $N$ is of type $\Clg{m-1}$. Then we show that $\pi\circ\varphi$ is homotopic to the identity. And then we use $\chi\in \TopS^m(G;\ZZ)$ to show that $\pi\circ \varphi$ sends cycles to boundaries.

The first part of the proof provides us with $n_0\le n_1\le\cdots \le n_{m-1}$ and for every $k\ge n_m$  a chain homomorphism 
\[
\varphi_*:{\chains * k {Q_{\bar \chi}}}^{(m-1)}\to {\chains * k {G_\chi}}^{(m-1)}
\]
extending the identity on $\ZZ$ with $\im \varphi\subseteq \chains q {n_q} {G_\chi}$ for $q=0,\ldots,m-1$.

We will define $l_0\le\cdots\le l_{m-1}$ and construct a chain homotopy $\mu_*:{\chains * k {Q_{\bar\chi}}}^{(m-1)}\to {\chains{*+1} k {Q_{\bar\chi}}}^{(m)}$ that joins $\pi\circ \varphi$ to $\id$ with $\im \mu_q\subseteq \chains {q+1}{l_q}{Q_{\bar \chi}}$ for $q=-1,\ldots,m-1$. As a shorthand for $x_0^{-1}\mu_q(x_0,\ldots,x_q)$, we write $u_{0\cdots q}$. We define $\mu_{-1}:=0$. Suppose $\mu_{-1},\ldots,\mu_{q-1}$ have been constructed. Part of the induction hypothesis is that $u_{0\cdots q-1}\in \chains q {l_{q-1}} {\Delta_{k(q-1)}[1_G]}$. Let $(x_0,\ldots,x_q)\in \chains q k {Q_{\bar\chi}}$ be a simplex. Then
\begin{align}
\label{eq:ge121}
 \id(x_0,\ldots,x_q)\in x_0\chains q k {\Delta_k[1_Q]}
\end{align}
and
\begin{align}
\label{eq:ge122}
\begin{split}
 \pi\circ \varphi(x_0,\ldots,x_q)\in x_0\chains q {n_q} 
{\pi(\Delta_{kq}[N])}=x_0 \chains q {n_q} {\Delta_{kq}[1_Q]}.
\end{split}
\end{align}
By induction hypothesis
\begin{align}
\label{eq:ge123}
\begin{split}
 x_0^{-1}\mu_{q-1}(x_1,\ldots,x_q)
 &=x_0^{-1}x_1u_{1\cdots q}\\
 &\in x_0^{-1}x_1\chains q {l_{q-1}}{\Delta_{k(q-1)}[1_Q]}\\
 &\subseteq\chains q {l_{q-1}}{\Delta_{kq}[1_Q]}.
 \end{split}
\end{align}
Without loss of generality, we assume $l_{q-1}\ge \max(2qk,n_q)$. Then equation~\ref{eq:ge121}, equation~\ref{eq:ge122} and equation~\ref{eq:ge123} combine to
\[
 x_0^{-1}(\id-\pi\circ \varphi-\mu_{q-1}\circ\diff q)(x_0,\ldots,x_q)\in 
\chains q {l_{q-1}} {\Delta_{qk}[1_Q]}=\chain_q(\E\Delta_{qk}[1_Q]).
\]
(In this paper $\E \cdot$ denotes the free simplicial set, the homology of which is trivial.) By basic computations $\diff q \circ (\id-\pi\circ \varphi-\mu_{q-1}\circ\diff q)=0$. Then there is $u_{0\cdots q}\in \chain_{q+1}(\E\Delta_{qk}[1_Q])$ with 
\[
 \diff {q+1} u_{0\cdots q}=x_0^{-1}(\id-\pi\circ \varphi-\mu_{q-1}\circ\diff 
q)(x_0\ldots,x_q).
\]
Define $\mu_q(x_0,\ldots,x_q):=x_0u_{0\cdots q}$ and $l_q:=2qk$.

Suppose we did already show that $\bar\chi\in \TopS^{m-1}(Q;\ZZ)$. Let $z\in \chains {m-1} k {Q_{\bar\chi}}$ be a cycle. Then $\varphi_{m-1}(z)\in \chains {m-1} {n_{m-1}} {G_\chi}$ is also a cycle. Since $\chi\in\TopS^m(G;\ZZ)$ there exists some $l\ge n_{m-1}$ such that $\homology {m-1} {n_{m-1}} {G_\chi}$ vanishes in $\homology {m-1} l {G_\chi}$. So there exists a chain $c\in \chains m l {G_\chi}$ with $\diff m c =\varphi_{m-1}(z)$. Then
\[
\diff m(\pi(c)+\mu_{m-1}(z))= \partial_m \pi(c)+\diff m\circ 
\mu_{m-1}(z)=\pi\circ \varphi(z)+z-\pi\circ \varphi(z)=z.
\]
In this way, $\homology {m-1} k {Q_{\bar\chi}}$ vanishes in $\homology {m-1} {\max(l,l_{m-1})}{Q_{\bar\chi}}$. Thus, $\bar\chi\in \TopS^m(Q;\ZZ)$. 
\end{proof}

%% file: ge2.tex
\section{Group extensions of abelian quotients}
This section proves Theorem~\ref{thma:ge2}.

 The setting of this section is a short exact sequence of locally compact Hausdorff groups
\[
 1\to N\to G\to Q\xrightarrow{\pi} 0
\]
where $Q$ is abelian. We define 
 \[
  S(G,N):=\{\chi\in \TopHom(G,\RR)\mid \chi|_N=0\}.
 \]
 Since $Q$ is an abelian, locally compact, compactly generated group, it is of the form $Q=\RR^l\times \ZZ^n\times Q_t$ where $Q_t$ is compact \cite[Theorem~24]{Morris1977}. First, we reduce to the case where $Q$ does not have a compact factor. The group $\pi^{-1}(Q_t)=:N_t$ contains $N$ as a cocompact subgroup, and $Q_1:=G/N_t=\RR^l\times \ZZ^n$ is of the required form. Since $N$ and $N_t$ have the same geometry, one of them being of type $\Clg m$ implies that the other is of type $\Clg m$ as well. So, from now on, we assume $Q=\RR^l\times \ZZ^n$. 
 
 Denote by $\langle\cdot,\cdot\rangle$ the standard scalar product in $\RR^{l+n}$ and by $\Vert\cdot\Vert$ the standard norm on $\RR^{l+n}$. There is a canonical mapping
 \begin{align*}
  G\setminus N&\mapsto S(G,N)\\
  g&\mapsto \chi_g:h\mapsto -\frac{\langle \pi(h),\pi(g)\rangle}{\Vert \pi(g)\Vert },
 \end{align*}
and $v_g$ denotes the valuation on $\chains q k G$ extending $\chi_g$. The norm $\Vert g\Vert:=\Vert\pi(g)\Vert$ is actually not a norm in the strict sense, but we can consider it as the distance of $g$ to $N$. We can make this precise:

\begin{lem}
\label{lem:compare}
 In the above setting
 \begin{enumerate}
  \item $\ell(\pi(g))=d(g,N)$, where $l$ is with respect to $\pi(\mathcal X)$ and $d$ is with respect to $\mathcal X$, which is a compact neighborhood and generating set for $G$;
  \item there is $K\ge 0$, such that 
  \[
   \Vert g\Vert\le K\cdot \ell(\pi(g))\, \forall g\in G;
  \]
 \item there is $\varepsilon >0$, such that 
 \[
  \ell(\pi(g))\le \tfrac {\Vert g\Vert}\varepsilon+1\, \forall g\in G;
 \]
 \item If $A\ge 0$ then 
 \[
  N^+:=\{g\in G\mid \Vert g\Vert \le A\}
 \]
has the same coarse type as $N$.
 \end{enumerate}
\end{lem}
\begin{proof}
We first show claim 1. We have $l(\pi(g))\le n$ if $g$ can be written as $g=x_1\cdots x_n d$ where $d\in N$ and $x_1,\ldots,x_n\in \mathcal X$. Now $N$ is normal, so $g=d'x_1\cdots x_n$ with $d'\in N$. So $d(g,N)\le n$. The reverse implication also holds.
  
  Now we show claim 2. Since $\pi(\mathcal X)$ is a compact neighborhood of $0_Q$ there exists $K\ge 0$ with 
  \[
   x\in \pi(\mathcal X)\implies \Vert x\Vert \le K
  \]
If $q\in Q$ with $l(q)=n$, then we can write $q=x_1+\cdots +x_n$ where $x_i\in \pi(\mathcal X)$. Then 

\[
 \Vert q\Vert =\Vert x_1+\cdots +x_n\Vert\le \Vert x_1\Vert+\cdots +\Vert x_n\Vert \le K\cdot l(q).
\]

Now we show claim 3. If $q\in Q=\RR^l\times \ZZ^n$, then we write $q=q_1+q_2$ with $q_1\in \RR^l\times \{0\}$ and $q_2\in \{0\}\times \ZZ^n$. 

We first look at the real factor: Since $\pi(\mathcal X)\cap (\RR^l\times 0)$ is a compact neighborhood of $0_{\RR^l}$ there exists $\varepsilon'>0$ with 
\[
 \Vert y\Vert \le \varepsilon' \implies y\in \pi(\mathcal X)\cap (\RR^l\times 0).
\]
Note that $\mathcal Y_1:=\{y\in Q\mid \Vert y\Vert \le \varepsilon'\}$ also forms a generating set for $\RR^l$, the word length of which we denote by $l_1$. We have
\[
 l(q_1)\le l_1(q_1)\le\tfrac{\Vert q_1\Vert}{\varepsilon'}+1. 
\]
To see the last inequality, we can draw a line through $0$ and $q_1$ and align points $y_i$ on the line, each at distance $i\cdot \varepsilon'$ from $0$.

Now we look at the integral factor: Define $\mathcal Y_2:=\{(0,\ldots,0,1,0,\ldots,0)\in \ZZ^n\}\cup\{(0,\ldots,0,-1,0,\ldots,0)\in\ZZ^n\}$. The word length according to $\mathcal Y_2$ is denoted by $l_2$. Then there exists some $K>0$ with $l(y)\le K\cdot l_2(y)$ for every $y\in Q$. Now $l_2$ is the $\ell_1$-norm on $\ZZ^n$, so we have
\[
 l(q_2)\le K\cdot l_2(q_2)\le K\sqrt n \Vert q_2\Vert.
\]

Then 
\begin{align*}
 l(q)
 \le& l(q_1)+l(q_2)\\
 \le& \tfrac{\Vert q_1\Vert}{\varepsilon'}+1+\Vert q_2\Vert\cdot K\sqrt n\\
 \le & \max(\tfrac 1 {\varepsilon'},K\sqrt n)(\Vert q_1\Vert+\Vert q_2\Vert)+1\\
 \le & 2 \max(\tfrac 1 {\varepsilon'},K\sqrt n)\Vert q\Vert+1.
\end{align*}
We now set $\varepsilon:=\frac 1 { 2 \max(\tfrac 1 {\varepsilon'},K\sqrt n)}$.

Now we show claim 4. Since 
\[
 N\subseteq N^+\subseteq \{g\in G\mid l(\pi(g))\le \tfrac A \varepsilon +1\}=\Delta_{\tfrac A \varepsilon +1}[N]
\]
and $N$ and $\Delta_{\tfrac A \varepsilon +1}[N]$ have the same coarse type (via inclusion), that is, Lemma~\ref{lem:coarse=easy}, the intermediate space $N^+$ has the same coarse type as well.
\end{proof}

The mapping $\Vert \cdot \Vert$ on $G$ can be extended to a norm on chains (which is also not a norm in the strict sense): If $c\in \chains q k G$, then
\[
 \Vert c\Vert:=\begin{cases}
         -\infty & c=0\\
         \max\{\Vert g\Vert \mid g\in c^{(0)}\} &c\not=0.
        \end{cases}
\]
  \begin{thm}
  \label{thm:ge2}
  Let $N\trianglelefteq G$ be a closed normal subgroup with abelian factor group $\pi:G\to Q=G/N$. Then $N$ is of type $\Clg m$ if $S(G,N)\subseteq \TopS^m(G;\ZZ)$.
  \end{thm}
  \begin{proof}
 The structure of the proof is very similar to one direction of \cite[Theorem~5.1]{Bieri1988}.
 
Suppose that we have already shown that $N$ is of type $\Clg {m-1}$. Let $k\ge 0$ be a number and suppose that $\homology {m-1} k G$ vanishes in $\homology {m-1} l G$ for some $l\ge k$. 

 If $\sim$ denotes the relation in $\TopHom(G,\RR)$ such that $\chi\sim \psi$ if $\chi=r\psi$ for some $r>0$ then $(S(G,N)\setminus\{0\})/{\sim}\cong S^{l+n-1}$ is a compact subset of $(\TopS^m(G,\ZZ)\setminus{(0)})/{\sim}$. Each $\chi\in S(G,N)\setminus \{0\}$, since it belongs to $\TopS^m$, is assigned a finitely modeled chain endomorphism $\varphi_\chi:{\chains * l G}^{(m)}\to {\chains * l G}^{(m)}$ extending the identity on $\ZZ$, which witnesses, that $\chi\in \TopS^m$ by applying Theorem~\ref{thm:crit2}. Then $U(\varphi_\chi)=u_{\varphi_\chi}^{-1}(0,\infty)$ is as we recall from Section~\ref{sec:open} an open subset contained in $\TopS^m$. Then $(U(\varphi_\chi)/{\sim})_{\chi\in S(G,N)\setminus \{0\}}$ forms an open cover of $(S(G,N)\setminus \{0\})/{\sim}$ which, since $(S(G,N)\setminus \{0\})/{\sim}$ is compact, has a finite subcover. 
\[
 U(\varphi_1)\cup \cdots \cup U(\varphi_n)\supseteq S(G,N)\setminus \{0\}.
\]

If $\chi\in S(G,N)\setminus \{0\}$, then there is some $i\in \{1,\ldots,n\}$, with $\chi\in U(\varphi_i)=u_{\varphi_i}^{-1}(0,\infty)$, so $u_{\varphi_i}(\chi)>0$. This way
\[
 \rho(\chi):=\max_{i=1,...,n}u_{\varphi_i}(\chi)>0.
\]
is a continuous function with positive real values. Now we identify $(S(G,N)\setminus\{0\})/{\sim}$ with the sphere $S^{l+n-1}$, which is a collection of representatives, each of norm $1$. Since this is a compact space, the infimum
\[
 r:=\inf\{\rho(\chi)|\chi\in S(G,N)\}>0
\]
is attained as a minimum.

Suppose $\eta_i$ is the homotopy of Lemma~\ref{lem:homotopy_homology_locallycompact} joining $\varphi_i$ to the identity. Then there exists $s_i\ge 0$ with $l(x)\le s_i$ for every $x\in \eta_i(\sigma)^{(0)}$ for every $\sigma=(1_G,g_1,\ldots,g_q)\in \Delta^q_l$. This holds for all $i=1,\ldots,n$, so we have
\begin{align}
\label{eq:s}
 l(x)\le \max_i(s_i)=:s.
\end{align}
Define $A:=\frac{K^2(s+l)^2+1}{2r}$ and $N^+:=\{g\in G\mid \Vert g \Vert\le A\}$. Let $z\in\chains{m-1}k{N^+}$ be a cycle. Then since $G$ is of type $\Clg m$ there exists $c\in \chains m l G$ with $\diff m c=z$. We now give a short outline of the rest of the proof. If $\Vert c\Vert>A$ we want to replace the chain $c$ by a chain $\bar c$ of smaller norm with the same boundary. For that, we consider the following procedure: We choose $g\in c^{(0)}$ where the maximum $\Vert c\Vert =:a$ is attained. Then we modify $c$ so as to remove this element from the support at the expense of introducing new elements $h$ with $\Vert h\Vert <\sqrt{a^2-1}$. This reduces the number of elements with the maximum norm $a$ in $c^{(0)}$ by one. Since $c^{(0)}$ is finite, repeating this procedure eventually yields a chain $\bar c$ with smaller norm $\Vert\bar c\Vert<a$, where the new vertices have norm $<\sqrt{a^2-1}$. We only need to reduce the norm at most $\Vert c\Vert^2$ times to obtain a chain $\bar c$ with $\Vert\bar c\Vert\le A$.

So, if $g\in c^{(0)}$ with $\Vert c\Vert =\Vert g\Vert=a$ we write $c=c'+c''$ where $c'$ collects all simplices $\sigma$ in $c$ with $gd\in \sigma^{(0)}$ for some $d\in N$. There exists some $i\in\{1,\ldots,n\}$ with $\chi_g\in U(\varphi_i)$. 

Define 
\begin{align*}
 \bar c
 :=&c+\diff {m+1} \eta_i(c')\\
 =&c-\eta_i\diff m(c')+\varphi_i(c')-c'\\
 =&\varphi_i(c')-\eta_i\diff m(c')+c''.
\end{align*}
This chain has the same boundary as $c$:
\[
\partial_m \bar 
c=\partial_m(c+\partial_{m+1}\eta_i(c'))=\partial_mc+\partial_m\circ\partial_{m+1
}\circ\eta_i(c')=\partial_m c.
\]
Then the character
\[
 \chi_g:h\mapsto -\frac{\langle \pi(h),\pi(g)\rangle}{\Vert\pi(g)\Vert}
\]
takes its minimum value in $c^{(0)}$ precisely at the elements of the form $gd$, $d\in N$ and 
\[
 \chi_g(gd)=-\frac{\langle \pi(gd),\pi(g)\rangle}{\Vert\pi(g)\Vert}=-\Vert\pi(g)\Vert=-a.
\]
Since $c''^{(0)}\subseteq c^{(0)}$ does not contain such elements, we obtain the following. 
\begin{align}
 \label{eq:vg3}
 v_g(c'')>-a.
\end{align}
 We have $v_g(\varphi_i(f))-v_g(f)\ge r$ for every $f\in\chains m l G$. So,
\begin{align}
\label{eq:vg1}
\begin{split}
 v_g(\varphi_i(c'))
 &\ge v_g(c')+r\\
 &\ge -a+r\\
 &>-a.
 \end{split}
\end{align}
 Since $\Vert\pi(z)\Vert<a$, we obtain for every $x\in z^{(0)}$:
\[
 |\chi_g(x)|=\frac{|\langle \pi(x),\pi(g)\rangle|}{\Vert \pi(g)\Vert}\le \Vert \pi(x)\Vert<a
\]
by the Cauchy--Schwarz inequality. So $v_g(z)>-a$. The Lemma~\ref{lem:homotopy_homology_locallycompact} implies that $(\eta_i(f))^{(0)}\subseteq f^{(0)}\cup (\varphi_i(f))^{(0)}$ for every $f\in \chains {m-1} l G$. Thus, 
\begin{align}
 \label{eq:vg2}
 \begin{split}
 v_g(\eta_i\circ\diff m(c'))
 &\ge \min(v_g(\diff m c'),v_g(\varphi_i(\diff m c')))\\
 &= \min(v_g(z-\diff m c''),v_g(\varphi_i(\diff m c')))\\
 &\ge\min(v_g(z),v_g(\diff m c''),v_g(\varphi_i(\diff m c'))\\
 &>-a.
 \end{split}
\end{align}
So, inequalities~\ref{eq:vg1},\ref{eq:vg2},\ref{eq:vg3} combine to the inequality
\begin{align*}
 v_g(\bar c)
 &=v_g(\varphi_i(c')-\eta_i\diff m(c')+c'')\\
 &\ge \min(v_g(\varphi_i(c')),v_g(\eta_i\diff m(c')),v_g(c''))\\
 &>-a.
\end{align*}
If $gd$ with $d\in N$ was a vertex in $\bar c$, then $v_g(\bar c)\le \chi_g(gd)=-a$, which contradicts $v_g(\bar c)>-a$. So we have shown that $dg\not\in \bar c^{(0)}$ for every $d\in N$.

\begin{figure}[t!]
	\centering
	\def \svgwidth{0.7\linewidth}
	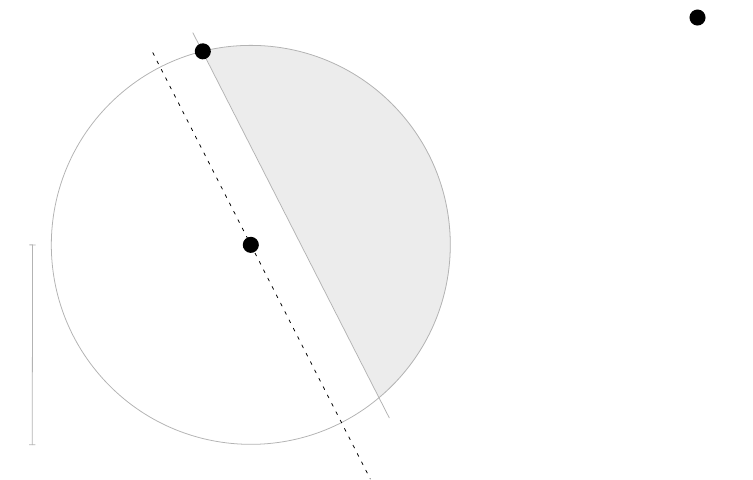
	\caption{bound on $\Vert h\Vert$ derived from the other parameters}
	\label{fig:pic2}
\end{figure}

Now we show that the new vertices, the ones contained in $\bar c$ but not in $c$, all have norm $<\sqrt{a^2-1}$. If $h\in G$ is a new vertex, then $h\in (\diff {m+1} \eta_i(c'))^{(0)}\subseteq (\eta_i(c'))^{(0)}$. Then $g^{-1}h\in g^{-1}(\eta_i(c'))^{(0)}=(\eta_i(g^{-1}c'))^{(0)}$. If $\sigma=(x_0,\ldots,g,\ldots,x_m)\in c'$, then $g^{-1}\sigma=g^{-1}x_0(1_G,\ldots,x_0^{-1}g,\ldots,x_0^{-1}x_m)$. Then inequality~\ref{eq:s} implies $l(x)\le s$ for $x\in (\eta_i(1_G,\ldots,x_0^{-1}g,\ldots,x_0^{-1}x_m))^{(0)}$. So $l(\pi(g^{-1}h))\le l(g^{-1}h)\le l(g^{-1}x_0)+l(x)\le l +s$. This implies
\begin{align}
\label{eq:norm1}
 \Vert\pi(g)-\pi(h)\Vert=\Vert \pi(g^{-1}h)\Vert\le K\cdot l(\pi(g^{-1}h))\le K\cdot(s+l).
\end{align}
Also $(\diff m \eta_i(c'))^{(0)}\subseteq (\eta_i(c'))^{(0)}\subseteq c'^{(0)}\cup(\varphi(c'))^{(0)}$, so $h\in(\varphi(c'))^{(0)}$, which implies
\begin{align}
 \label{eq:norm2}
  v_g(h)\ge-a+r.
 \end{align}
 Inequalities~\ref{eq:norm1},\ref{eq:norm2} imply that $h$ is in the marked region in Figure~\ref{fig:pic2}. Now Pythagoras' theorem implies
\[
 \Vert h\Vert\le ((a-r)^2+K^2(s+l)^2-r^2)^{1/2}=(a^2-2ar+K^2(s+l)^2)^{1/2}<\sqrt{a^2-1}
\]
by the assumption that $a>A=\frac{K^2(s+l)^2+1}{2r}$.

In this way, we have shown that we can replace $c$ by $\bar c$ of a smaller norm, as long as $\Vert c\Vert > A$. So after finitely many steps we reach $\bar{\bar c}\in \chains m {\max(l,2s)} {N^+}$. So $\homology {m-1} k {N^+}$ vanishes in $\homology {m-1} {\max(2s,l)}{N^+}$. Then Lemma~\ref{lem:compare} implies that $N$ is also of type $\Clg m$.
  \end{proof}

%% file: Figures/pic3.pdf_tex
\begingroup%
  \makeatletter%
  \providecommand\color[2][]{%
    \errmessage{(Inkscape) Color is used for the text in Inkscape, but the package 'color.sty' is not loaded}%
    \renewcommand\color[2][]{}%
  }%
  \providecommand\transparent[1]{%
    \errmessage{(Inkscape) Transparency is used (non-zero) for the text in Inkscape, but the package 'transparent.sty' is not loaded}%
    \renewcommand\transparent[1]{}%
  }%
  \providecommand\rotatebox[2]{#2}%
  \newcommand*\fsize{\dimexpr\f@size pt\relax}%
  \newcommand*\lineheight[1]{\fontsize{\fsize}{#1\fsize}\selectfont}%
  \ifx\svgwidth\undefined%
    \setlength{\unitlength}{361.37365867bp}%
    \ifx\svgscale\undefined%
      \relax%
    \else%
      \setlength{\unitlength}{\unitlength * \real{\svgscale}}%
    \fi%
  \else%
    \setlength{\unitlength}{\svgwidth}%
  \fi%
  \global\let\svgwidth\undefined%
  \global\let\svgscale\undefined%
  \makeatother%
  \begin{picture}(1,0.63669053)%
    \lineheight{1}%
    \setlength\tabcolsep{0pt}%
    \put(0,0){\includegraphics[width=\unitlength,page=1]{pic3.pdf}}%
    \put(-0.0020673,0.17105397){\color[rgb]{0.50196078,0.50196078,0.50196078}\makebox(0,0)[lt]{\lineheight{1.25}\smash{\begin{tabular}[t]{l}$K(s+l)$\end{tabular}}}}%
    \put(0,0){\includegraphics[width=\unitlength,page=2]{pic3.pdf}}%
    \put(0.27423893,0.30330142){\color[rgb]{0,0,0}\makebox(0,0)[lt]{\lineheight{1.25}\smash{\begin{tabular}[t]{l}$g$\end{tabular}}}}%
    \put(0.31367097,0.54134461){\color[rgb]{0,0,0}\makebox(0,0)[lt]{\lineheight{1.25}\smash{\begin{tabular}[t]{l}$h$\end{tabular}}}}%
    \put(0.70619808,0.42301641){\color[rgb]{0,0,0}\makebox(0,0)[lt]{\lineheight{1.25}\smash{\begin{tabular}[t]{l}$a$\end{tabular}}}}%
    \put(0.40990033,0.24450525){\color[rgb]{0,0,0}\makebox(0,0)[lt]{\lineheight{1.25}\smash{\begin{tabular}[t]{l}$r$\end{tabular}}}}%
    \put(0.95453143,0.6057843){\color[rgb]{0,0,0}\makebox(0,0)[lt]{\lineheight{1.25}\smash{\begin{tabular}[t]{l}$1_G$\end{tabular}}}}%
    \put(0,0){\includegraphics[width=\unitlength,page=3]{pic3.pdf}}%
  \end{picture}%
\endgroup%